\documentclass[english]{amsart}
\usepackage[applemac]{inputenc}
\usepackage{latexsym}

\usepackage{amsmath, mathptm, amsfonts, amssymb, verbatim, psfrag, graphicx, mathrsfs, xypic}

\theoremstyle{plain}
\newtheorem{theo}{Theorem}[section]
\newtheorem{lemm}[theo]{Lemma}
\newtheorem{prop}[theo]{Proposition}
\newtheorem{coro}[theo]{Corollary}
\newtheorem{thm}{Theorem}
\theoremstyle{remark}

\theoremstyle{definition}
\newtheorem{defi}[theo]{Definition}
\newtheorem{defis}[theo]{Definitions}

\newtheorem{rema}[theo]{Remark}
\newtheorem{ques}[theo]{Question}

\newtheorem{egs}[theo]{Examples}

  \usepackage[all]{xy}
\xyoption{matrix}
\newcommand{\PGL}{\mathrm{PGL}}
\newcommand{\im}{{\bf i}}
\newcommand{\Bir}{\mathrm{Bir}}
\newcommand{\Aut}{\mathrm{Aut}}

\newcommand{\C}{\mathbb{C}}
\newcommand{\p}{\mathbb{P}}
\newcommand{\z}{\mathbb{Z}}

\newcommand{\SL}{\mathrm{SL}}  
\newcommand{\PSL}{\mathrm{PSL}}

\newcommand{\Pic}[1]{\mathrm{Pic}(#1)}

\usepackage[colorlinks, linktocpage, 
citecolor = blue, linkcolor = blue]{hyperref}

\title{Embeddings of $\mathrm{SL}(2,\z)$ into the Cremona group}

\author{J\'er\'emy Blanc}
\address{J\'er\'emy Blanc, Universit\"{a}t Basel, Mathematisches Institut, Rheinsprung $21$, CH-$4051$ Basel, Switzerland.}
\email{jeremy.blanc@unibas.ch}
\author{
Julie D\'eserti}
\address{Julie D\'eserti, Universit\"{a}t Basel, Mathematisches Institut, Rheinsprung $21$, CH-$4051$ Basel, Switzerland. {\it On leave from} Institut de Math\'ematiques de Jussieu, Universit\'e Paris $7$. France.}

\email{julie.deserti@unibas.ch, deserti@math.jussieu.fr}

\thanks{Both authors supported by the Swiss National Science Foundation grant no PP00P2\_128422 /1.}

\begin{document}
\maketitle\begin{center}{\today}\end{center}

\begin{abstract}
Geometric and dynamic properties of embeddings of $\mathrm{SL}(2,\mathbb{Z})$ into the Cremona group are studied. Infinitely many non-conjugate embeddings that preserve the type (i.e. that send elliptic, parabolic and hyperbolic elements onto elements of the same type) are provided. The existence of infinitely many non-conjugate elliptic, parabolic and hyperbolic embeddings is also shown. 

In particular, a group $G$ of automorphisms of a smooth surface $S$ obtained by blowing-up $10$ points of the complex projective plane is given. The group $G$ is isomorphic to $\mathrm{SL}(2,\mathbb{Z})$, preserves an elliptic curve and all its elements of infinite order are hyperbolic.

\end{abstract}

\section{Introduction}

Our article is motivated by the following result on the embeddings of the groups~$\SL(n,\z)$ into the group $\mathrm{Bir}(\mathbb{P}^2)$ of birational maps of $\mathbb{P}^2(\mathbb{C})$: the group $\SL(n,\z)$ does not embed into $\mathrm{Bir}(\mathbb{P}^2)$ for $n\ge 4$ and  $\SL(3,\z)$ only embeds linearly (\emph{i.e.} in $\Aut(\mathbb{P}^2)=\mathrm{PGL}(3,\mathbb{C})$) into $\mathrm{Bir}(\mathbb{P}^2)$ up to conjugacy \cite[Theorem 1.4]{De}. 

It is thus natural to look at the embeddings of $\SL(2,\z)$ into $\mathrm{Bir}(\mathbb{P}^2)$. As $\SL(2,\z)$ has almost a structure of free group, it admits many embeddings of different type into $\mathrm{Bir}(\mathbb{P}^2)$, and it is not reasonable to look for a classification of \emph{all} embeddings. We thus focus on embeddings having some geometric properties; among them the most natural ones are the embeddings which \emph{ preserve the type} evoked by Favre in \cite[Question~4]{Fa}. 

The elements of $\SL(2,\z)$ are classified into elliptic, parabolic and hyperbolic elements, with respect to their action on the hyperbolic upper-plane (or similarly to their trace, \emph{see}~$\S\ref{SubSec:SL2Z}$). The Cremona group $\mathrm{Bir}(\mathbb{P}^2)$ naturally acts on a hyperbolic space of infinite dimension (\emph{see} \cite{Ma,Ca}), so there is a notion of elliptic, parabolic and hyperbolic elements in this group; this classification can also be deduced from the growth rate of degrees of iterates (\emph{see} \cite{DiFa} and $\S \ref{SubSec:Dynamic}$). Note that some authors prefer the term of loxodromic elements instead of hyperbolic elements (\emph{see for example} \cite[Proposition 2.16]{Andersen}).
 A morphism from $\SL(2,\z)$ to  $\mathrm{Bir}(\mathbb{P}^2)$ \emph{preserves the type} if it sends elliptic, parabolic and hyperbolic elements of $\SL(2,\z)$ on elements of~$\mathrm{Bir}(\mathbb{P}^2)$ of the same type. Up to now, the only known example is the classical embedding $\theta_s\colon\mathrm{SL}(2,\mathbb{Z})\to\mathrm{Bir}(\mathbb{P}^2)$, which associates to a matrix $M=~\left[
\begin{array}{cc}
a & b\\
c & d
\end{array}
\right]$ the birational map $\theta_s(M)$, given in affine coordinates by $(x,y)\dasharrow(x^ay^b,x^cy^d)$ (or written simply $(x^ay^b,x^cy^d)$). 
 In this article, we provide infinitely many non-conjugate embeddings that preserve the type (Theorem~$\ref{thm:preserve}$ below):

 Recall that the group $\mathrm{SL}(2,\mathbb{Z})$ is generated by the elements $R$
and $S$ given by
\begin{align*}
& R=\left[\begin{array}{cc}
1 & 1\\
0 & 1
\end{array}\right] && \text{and} && S=\left[\begin{array}{cc}
0 & 1\\
-1 & 0
\end{array}
\right].
\end{align*}

\begin{thm}$($see $\S\ref{Sec:Embeddtype})$ \label{thm:preserve}
Let $\varepsilon$ be a real positive number, and set
\begin{align*}
&\theta_{\varepsilon}(S)=\left(y,-x\right), && \theta_{\varepsilon}(R)=\left(\frac{x+\varepsilon y}{\varepsilon+xy}, \varepsilon y\right).
\end{align*}
Then $\theta_\varepsilon$ is an embedding of $\mathrm{SL}(2,\mathbb{Z})$ into the Cremona group that preserves the type.

Furthermore, if $\varepsilon$ and $\varepsilon'$ are two real positive numbers such that $\varepsilon\varepsilon'\not=1$, then $\theta_\varepsilon(\mathrm{SL}(2,\mathbb{Z}))$ and $\theta_{\varepsilon'}(\mathrm{SL}(2,\mathbb{Z}))$ are not conjugate in $\mathrm{Bir}(\mathbb{P}^2)$. 

The standard embedding $\theta_s$ is conjugate to $\theta_1$.
\end{thm}

This family of embeddings is a first step in the classification of all embeddings of~$\SL(2,\z)$ preserving the type. We do not know if other embeddings exist (except one special embedding $\theta_{-}$ described in $\S\ref{ThetaMo}$ which is a "twist" of the standard embedding $\theta_s$ defined by: $\theta_{-}(S)=\theta_s(S)=(y,\frac{1}{x})$ and $\theta_{-}(R)=(xy,-y)\not=\theta_s(R)=(xy,y)$), in particular if it is possible to find an embedding where the parabolic elements act by preserving elliptic fibrations.

\begin{ques}\label{Ques:OtherEmbedd}
Does there exist an embedding of ~$\SL(2,\z)$ into $\Bir(\p^2)$ that preserves the type and which is not conjugate to $\theta_{-}$ or to some $\theta_{\varepsilon}$?
\end{ques}

The last two assertions of  Theorem~\ref{thm:preserve} yield to the following question: 
\begin{ques}
Is the embedding $\theta_{-}$ rigid? \emph{i.e.} not extendable to a one parameter family of non conjugate embeddings ?
\end{ques}

Note that some morphisms $\SL(2,\z)\to \Bir(\p^2)$ preserving the type have been described (\cite[page~9]{Fa}, \cite{CaLo} and \cite{Go}), but that these ones are not embedding, the central involution acting trivially. See $\S \ref{Subsec:Standard}$ for more details.

\bigskip

One can also consider elliptic, parabolic and hyperbolic embeddings of  $\SL(2,\z)$ into  $\mathrm{Bir}(\mathbb{P}^2)$. An embedding~$\theta$ of $\mathrm{SL}(2,\mathbb{Z})$ into the Cremona group is said to be \emph{elliptic} if each element of $\mathrm{im}\,\theta$ is elliptic; $\theta$ is \emph{parabolic} (respectively \emph{hyperbolic}) if each element of infinite order of $\mathrm{im}\,\theta$ is parabolic (respectively hyperbolic). 

In Sections~\ref{Sec:ReprTypEllipt},~\ref{Sec:ReprTypParab} and~\ref{Sec:ReprTypHyperZ}, we prove the existence of an infinite number of non-conjugate elliptic, parabolic and hyperbolic embeddings (\emph{see} Propositions~\ref{ellipticemb},~\ref{parabolicemb1},~\ref{parabolicemb2} and Corollary~\ref{CoroHyp}). It is  possible to find many other such embeddings; we only give a simple way to construct infinitely many of each family.

\bigskip

One can then ask if it is possible to find an embedding of $\SL(2,\z)$ into the Cremona group which is \emph{regularisable}, \emph{i.e.}\ which comes from an embedding into the group of automorphisms of a projective rational surface. It is easy to construct elliptic embeddings which are regularisable (see Section~\ref{Sec:ReprTypEllipt}). In Section~\ref{Sec:invfixbourbell}, we give a way to construct infinitely many hyperbolic embeddings of $\SL(2,\z)$ into the Cremona group which are re\-gularisable, and each of the groups constructed moreover preserves an elliptic curve (one fixing it pointwise). The existence of regularisable embeddings which preserve the type is still open (and should contain parabolic elements with quadratic growth of degree).

Note that the existence of hyperbolic automorphisms preserving an elliptic curve was not clear. In \cite[Theorem 1.1]{Pan}, it was proved that a curve preserved by an hyperbolic element of $\Bir(\p^2)$ has geometric genus $0$ or $1$;  examples of genus $0$ (easy to create by blowing-up) were provided, and the existence of genus $1$ curves invariant was raised (\emph{see}~\cite[page 443]{Pan}). The related question of the existence of curves of arithmetic genus~$1$ preserved by hyperbolic automorphisms of rational surfaces was also raised two years after in \cite[page 2987]{DJS}.
In \cite{Mc}, the author constructs hyperbolic automorphisms of rational surfaces which correspond to Coxeter elements (any hyperbolic automorphism of a rational surface corresponds to an element of the Weyl group associated to the surface), that preserve a cuspidal (resp.~nodal) curve. However, a general automorphism of a rational surface corresponding to a Coxeter element is hyperbolic but does not preserve any curve (\cite{BK2}).

The following statement yields existence of a group of automorphisms preserving a (smooth) elliptic curve such that every non-periodic element is hyperbolic. This is also possible with free groups (\emph{see} \cite[Remark 3.2]{CantatMathRes} and \cite{BlaMich}), but the construction is harder with more complicated groups like $\SL(2,\z)$. The method that we describe in Section~$\ref{Sec:invfixbourbell}$ should be useful to create other groups generated by elements of finite order. 

\begin{thm}\label{thm4}
There exist hyperbolic embeddings $\theta_{h,1},\theta_{h,2},\theta_{h,3}$ of $\mathrm{SL}(2,\mathbb{Z})$ into $\mathrm{Bir}(\mathbb{P}^2)$ such that:
\begin{itemize}
\item for each $i$, the group $\theta_{h,i}$ preserves a smooth cubic curve $\Gamma\subset \p^2$;
\item the action of $\theta_{h,1}$ on $\Gamma$ is trivial, the action of $\theta_{h,2}$ on $\Gamma$ is generated by a translation of order $3$ and the action of $\theta_{h,3}$ on $\Gamma$ is generated by an automorphism of order $3$ with fixed points;

\item for $i=1,2,3$, the blow-up $X_i\to \p^2$ of respectively $12,10,10$ points of $\Gamma$ conjugates $\theta_{h,i}(\mathrm{SL}(2,\mathbb{Z}))$ to a subgroup of automorphisms of $X_i$. The strict transform~$\widetilde{\Gamma}$ of $\Gamma$ on $X_i$ is the only invariant curve; in particular the orbit of any element of~$X_i\backslash \widetilde{\Gamma}$ is either finite or dense in the Zariski topology.
\end{itemize}

Moreover, in cases $i=1,2$, we can choose $\Gamma$ to be any smooth cubic curve, and this yields infinitely many hyperbolic embeddings of $\mathrm{SL}(2,\mathbb{Z})$ into $\mathrm{Bir}(\mathbb{P}^2)$, up to conjugacy.
\end{thm}
\begin{rema}
In $\theta_{h,1}$, $\theta_{h,2}$, $\theta_{h,3}$, the letter $h$ is no parameter but only means "hyperbolic", to distinguish them from the other embeddings $\theta_s$, $\theta_{-}$ and $\{\theta_{\epsilon}\}_{\epsilon\in \mathbb{R}}$, defined above.
\end{rema}
It could be interesting to study more precisely the orbits of the action of the above groups, in particular to answer the following questions:
\begin{ques}
Are the typical orbits of $\theta_{h,i}$ dense in the transcendental topology ?\end{ques}
\begin{ques}Are there some finite orbits in $X_i\backslash\widetilde{\Gamma}$?\end{ques}

We finish this introduction by mentioning related results.

\medskip

The statement of \cite[Theorem 1.4]{De} for $\SL(3,\mathbb{Z})$ was generalised in \cite{Ca}, where it is proven that any finitely generated group having Kazhdan's pro\-perty~(T) only embeds linearly into $\Bir(\p^2)$ (up to conjugation).
\medskip

Let us also mention \cite[Theorem A]{CaLa} which says that if a lattice $\Gamma$ of a simple Lie group~$G$ embeds into the group $\Aut(\C^2)$, then $G$ is isomorphic to $\mathrm{PSO}(1,n)$ or $\mathrm{PSU}(1,n)$ for some~$n$. If the embedding is not conjugate to a subgroup of the affine group, the only possibility is $G\simeq \mathrm{PSO}(1,2)\simeq \mathrm{PSL}(2,\mathbb{R})$, this latter case being intensively studied  in~\cite{CaLa}.

Note that our techniques heavily use the special structure of $\SL(2,\mathbb{Z})$, and one could ask similar questions for any lattice of $\mathrm{GL}(2,\mathbb{R})$ or $\mathrm{PGL}(2,\mathbb{R})$; the behaviour and results could be very different.

\medskip

{\it Acknowledgements.} The authors would like to thank Charles Favre  for interesting comments and suggestions, and Pierre de la Harpe for interesting discussions. Thanks also to the referees for their helpful remarks and corrections.

\section{Some reminders on $\mathrm{SL}(2,\mathbb{Z})$ and $\mathrm{Bir}(\mathbb{P}^2)$}\label{premieresproprietes}

\subsection{About $\mathrm{SL}(2,\mathbb{Z})$}\label{SubSec:SL2Z}

Division algorithm implies that the group $\mathrm{SL}(2,\mathbb{Z})$ is generated by the elements $R$
and $S$ given by
\begin{align*}
& R=\left[\begin{array}{cc}
1 & 1\\
0 & 1
\end{array}\right] && \text{and} && S=\left[\begin{array}{cc}
0 & 1\\
-1 & 0
\end{array}
\right].
\end{align*}

Remark that $R$ is of infinite order and $S$ of order $4$. The square of $S$ generates the center of $\mathrm{SL}(2,\mathbb{Z})$. Moreover
\begin{align*}
& RS=\left[\begin{array}{cc}
-1 & 1\\
-1 & 0
\end{array}
\right] && \text{and} && SR=\left[\begin{array}{cc}
0 & 1 \\
-1 & -1
\end{array}
\right]
\end{align*}
are conjugate by $S$ and both have order $3$.

A presentation of $\mathrm{SL}(2,\mathbb{Z})$ is given by $$\langle R,\,S\,\vert\, S^4=(RS)^3=1,\,S^2(RS)=(RS)S^2\rangle$$
(\emph{see for example}~\cite[Chapter 8]{Newman}). This implies that the quotient of $\mathrm{SL}(2,\mathbb{Z})$ by its center is a free product of $\mathbb{Z}/2\mathbb{Z}$ and $\mathbb{Z}/3\mathbb{Z}$
generated by the classes~$[S]$ of $S$ and $[RS]$ of $RS$ 
$$\mathrm{PSL}(2,\z)=\langle [S],\,[RS]\,\vert\, [S]^2=[RS]^3=1\rangle.$$

\subsection{Dynamic of elements of $\SL(2,\z)$}
Recall that the group $\SL(2,\mathbb{R})$ acts on the upper half plane
$$\mathbb{H}=\{x+ \im y\in \mathbb{C}\ |\ x,y\in \mathbb{R}, y>0\}$$  
by M\"obius transformations:
\begin{align*}
&\mathrm{SL}(2,\mathbb{R})\times \mathbb{H}\to \mathbb{H},
&& \left(\left[
\begin{array}{cc}
a & b\\
c & d
\end{array}
\right],z\right)\mapsto\frac{az +b}{cz+d}.
\end{align*}
The hyperbolic structure of $\mathbb{H}$ being preserved, this yields to a natural notion of \emph{elliptic}, \emph{parabolic}, and \emph{hyperbolic} elements of $\SL(2,\mathbb{R})$, and thus to elements of $\SL(2,\mathbb{Z})$ (as in \cite[II.8]{Iversen}).

\medskip

If $M$ is an element of $\mathrm{SL}(2,\mathbb{Z})$, we can be more precise and check the following easy observations:
\begin{itemize}
\item  $M$ is \emph{elliptic} if and only if $M$ has finite order;

\item $M$ is \emph{parabolic} (respectively \emph{hyperbolic}) if and only if $M$ has infinite order and its trace is $\pm 2$ (respectively $\not=\pm 2$).
\end{itemize}

Up to conjugacy the elliptic elements of $\mathrm{SL}(2,\mathbb{Z})$ are
\begin{align*}
&\left[\begin{array}{cc}
-1 & 0\\
0 & -1
\end{array}\right],
&&
\left[\begin{array}{cc}
0 & 1\\
-1 & -1
\end{array}\right],
&&
\left[\begin{array}{cc}
0 & 1\\
-1 & 0
\end{array}\right],
&&
\left[\begin{array}{cc}
0 & -1\\
1 & 0
\end{array}\right],
&&
\left[\begin{array}{cc}
0 & -1\\
1 & 1
\end{array}\right];
\end{align*}

in particular an element of finite order is of order $2$, $3$, $4$ or $6$.

\smallskip

A parabolic element of $\mathrm{SL}(2,\mathbb{Z})$ is up to conjugacy one of the following
\begin{align*}
&\left[\begin{array}{cc}
1 & a\\
0 & 1
\end{array}\right],&&
\left[\begin{array}{cc}
-1 & a\\
0 & -1
\end{array}\right], &&a\in\mathbb{Z}.
\end{align*}

\subsection{Cremona group and dynamic of its elements}\label{SubSec:Dynamic}
Let us recall the following classical definitions.

\begin{defis}
A \emph{ rational map} of the projective plane into itself is a map of the follo\-wing type
\begin{align*}
&f\colon\mathbb{P}^2(\mathbb{C})\dashrightarrow\mathbb{P}^2(\mathbb{C}), &&
(x:y:z)\dashrightarrow(f_0(x,y,z):f_1(x,y,z):f_2(x,y,z)),
\end{align*}

\noindent where the $f_i$'s are homogeneous polynomials of the same degree without common factor. The \emph{ degree} of $f$ is by definition: $\deg f=~\deg f_i$. A \emph{birational map} $f$ is a rational map that admits a rational inverse. We denote by $\mathrm{Bir}(\mathbb{P}^2)$ the group of birational maps of the projective plane into itself; $\mathrm{Bir}(\mathbb{P}^2)$ is also called the \emph{ Cremona group}.
\end{defis}

The degree is not a birational invariant; if $f$ and $g$ are in $\mathrm{Bir}(\mathbb{P}^2)$, then in general $\deg (gfg^{-1})\not=\deg f$ . Nevertheless there exist two strictly positive constants $a,b\in \mathbb{R}$ such that for all $n$ the following holds $$a\deg f^n\leq\deg (gf^ng^{-1})\leq b\deg f^n.$$ In other words the degree growth is a birational invariant; so we introduce the following notion (\cite{Fr, RS}). 

\begin{defi}
Let $f$ be a birational map. The \emph{ first dynamical degree} of  $f$ is defined by $$\lambda(f)=\lim (\deg f^n)^{1/n}.$$
\end{defi}

There is a classification of birational maps of $\p^2$ up to birational conjugation.

\begin{theo}[\cite{Gi, DiFa}]
Let $f$ be an element of $\mathrm{Bir}(\mathbb{P}^2)$. Up to birational conjugation, exactly one of the following holds. 
\begin{itemize}
\item The sequence $(\deg f^n)_{n\in \mathbb{N}}$ is bounded, $f$ is an automorphism on some projective rational surface and an iterate of~$f$ is an automorphism isotopic to the identity;

\item the sequence $(\deg f^n)_{n\in \mathbb{N}}$ grows linearly, and $f$ preserves a rational fibration. In this case $f$ cannot be conjugate to an automorphism of a projective surface;

\item the sequence $(\deg f^n)_{n\in \mathbb{N}}$ grows quadratically, and $f$ is conjugate to an automorphism preserving an elliptic fibration; 

\item  the sequence $(\deg f^n)_{n\in \mathbb{N}}$ grows exponentially.
\end{itemize}
\medskip

In the second and third case, the invariant fibration is unique. In the first three cases~$\lambda(f)$ is equal to $1$, in the last case $\lambda(f)$ is strictly greater than $1$.
\end{theo}

\begin{defis}
Let $f$ be a birational map of $\mathbb{P}^2$.

If the sequence $(\deg f^n)_{n\in \mathbb{N}}$ is bounded, $f$ is said to be \emph{ elliptic}.

When $(\deg f^n)_{n\in \mathbb{N}}$ grows linearly or quadratically, we say that $f$ is \emph{ parabolic}.

If $\lambda(f)>1$, then $f$ is an \emph{ hyperbolic map}.
\end{defis}

As we said the Cremona group acts naturally on a hyperbolic space of infinite dimension (\cite{Ma,Ca}); we can say that a birational map is elliptic, resp. parabolic, resp. hyperbolic, if the corresponding isometry is elliptic, resp. parabolic, resp. hyperbolic (\cite[Chapter 8, \S2]{DG}). This definition coincides with the previous one (\cite{Ca}).

\begin{egs}
Any automorphism of $\mathbb{P}^2$ or of an Hirzebruch surface $\mathbb{F}_n$ and any birational map of finite order is elliptic.

The map $(x:y:z)\dasharrow(xy:yz:z^2)$ is parabolic.

A H\'enon map (automorphism of $\C^2$)
\begin{align*}
& (x,y)\mapsto (y,P(y)-\delta x), && \delta\in\mathbb{C}^*,\, P\in\mathbb{C}[y],\,\deg P\geq 2
\end{align*}
extends to a hyperbolic birational map of $\mathbb{P}^2$, of dynamical degree~$\deg P$.
\end{egs}

\begin{defis}\label{Def:TypePlong}
Let $\theta\colon\mathrm{SL}(2,\mathbb{Z})\to\mathrm{Bir}(\mathbb{P}^2)$ be an embedding of $\mathrm{SL}(2,\mathbb{Z})$ into the Cremona group.

We say that $\theta$ \emph{ preserves the type} if $\theta$ sends elliptic (respectively parabolic, respectively hyperbolic) element onto elliptic (respectively parabolic, respectively hyperbolic) map.

We say that $\theta$ is \emph{ elliptic} if each element of $\mathrm{im}\,\theta$ is elliptic.

The morphism $\theta$ is \emph{ parabolic} (respectively \emph{ hyperbolic}) if each element of infinite order of $\mathrm{im}\,\theta$ is parabolic (respectively hyperbolic). 
\end{defis}

\bigskip

\subsection{The central involution  of $\mathrm{SL}(2,\mathbb{Z})$ and its image into $\mathrm{Bir}(\mathbb{P}^2)$}\label{involutioncentrale}

The element $S^2\in \mathrm{SL}(2,
\mathbb{Z})$ is an involution; therefore its image by any embedding $\theta\colon \mathrm{SL}(2,\mathbb{Z})\to \Bir(\p^2)$ is a birational involution. As it was proved by Bertini, we have the following classification:

\begin{theo}[\cite{Ber}]
An element of order $2$ of the Cremona group  is up to conjugacy one of the following 
\smallskip
\begin{itemize}
\item an automorphism of $\mathbb{P}^2$;

\item a de Jonqui\`eres involution $\iota_{\mathrm{dJ}}$ of degree~$\nu\geq 2$;

\item a Bertini involution $\iota_B$;

\item a Geiser involution $\iota_G$.
\end{itemize}
\end{theo}

\noindent Bayle and Beauville showed that the conjugacy classes of involutions in $\mathrm{Bir}(\mathbb{P}^2)$ are determined by the birational type of the curves of fixed points of positive genus (\cite{BaBe}). More precisely the set of conjugacy classes is parametrised by a disconnected algebraic variety whose connected components are respectively 
\begin{itemize}
\item the moduli spaces of hyperelliptic curves of genus $g$ (de Jonqui\`eres involutions);

\item the moduli space of canonical curves of genus $3$ (Geiser involutions);

\item the moduli space of canonical curves of genus $4$ with vanishing theta characteristic, isomorphic to a non singular intersection of a cubic surface and a quadratic cone in $\mathbb{P}^3(\mathbb{C})$ (Bertini involutions).
\end{itemize}

The image of $S^2$ can neither be a Geiser involution, nor a Bertini involution; more precisely, we have the following:

\begin{lemm}
Let $\theta$ be an embedding of $\mathrm{SL}(2,\mathbb{Z})$ into the Cremona group. Up to birational conjugation, one of the following holds.
\begin{itemize}
\item The involution $\theta(S^2)$ is an automorphism of $\mathbb{P}^2$;

\item the map $\theta(S^2)$ is a de Jonqui\`eres involution of degree~$3$ fixing $($pointwise$)$ an elliptic curve.
\end{itemize}
\end{lemm}

\begin{rema}
The first case is satisfied by the examples of \S\ref{Sec:Embeddtype}, \S\ref{Sec:ReprTypEllipt},  and \S\ref{Sec:ReprTypParab}. The second case is also possible, for any elliptic curve (see \S\ref{Sec:invfixbourbell}).
\end{rema}

\begin{proof}
Since $S^2$ commutes with $\mathrm{SL}(2,\mathbb{Z})$ the group $\mathrm{G}=\theta(\mathrm{SL}(2,\mathbb{Z}))$ is contained in the centraliser of the involution~$S^2$. If $\theta(S^2)$ is a Bertini or Geiser involution, the centraliser of~$\theta(S^2)$ is finite (\cite{BPV2}, Corollary $2.3.6$); as a consequence~$\theta(S^2)$ is a de Jonqui\`eres involution. 

Assume that $\theta(S^2)$ is not linearisable; then $\theta(S^2)$ fixes (pointwise) a unique irreducible curve $\Gamma$ of genus $\ge 1$. The group $\mathrm{G}$ preserves $\Gamma$ and the action of $\mathrm{G}$ on $\Gamma$ gives the exact sequence $$1 \to \mathrm{G}'\to\mathrm{G}\to \mathrm{H} \to 1$$ where $\mathrm{H}$ is a subgroup of $\mathrm{Aut}(\Gamma)$, $\mathrm{G}'$ contains $\theta(S^2)$ and fixes $\Gamma$. Since the genus of $\Gamma$ is positive $\mathrm{H}$ cannot be equal to~$\mathrm{G}/\langle \theta(S^2)\rangle$, free product of $\mathbb{Z}/2\mathbb{Z}$ and $\mathbb{Z}/3\mathbb{Z}$. This implies that the normal subgroup~$\mathrm{G}'$ of
$\mathrm{G}$ strictly contains $\langle \theta(S^2)\rangle$ and thus that it is infinite and not abelian. In particular the group of birational maps fixing (pointwise) $\Gamma$ is infinite, and not abelian, thus $\Gamma$ is of genus $1$ (\emph{see} \cite{BPV}, Theorem $1.5$).
\qed\end{proof}

\section{Embeddings preserving the type and elliptic, parabolic embeddings}

\subsection{Embeddings preserving the type}\label{Sec:Embeddtype} 
Henceforth we will often denote by $(f_1(x,y,z):f_2(x,y,z):f_3(x,y,z))$ the map $$(x:y:z)\dasharrow (f_1(x,y,z):f_2(x,y,z):f_3(x,y,z))$$
and by $(p(x,y),q(x,y))$ the birational map $$(x,y)\dasharrow (p(x,y),q(x,y))$$ of $\mathbb{C}^2$.

\medskip

Let us begin this section by a property satisfied by all embeddings of $\SL(2,\z)\to~\Bir(\p^2)$ that preserve the type.

\begin{lemm}
Let $\theta\colon\mathrm{SL}(2,\mathbb{Z})\to\mathrm{Bir}(\mathbb{P}^2)$ be an embedding that preserves the type. Either for all parabolic matrices~$M$, $\theta(M)$ preserves a unique rational fibration, or for all parabolic matrices $M$, $\theta(M)$ preserves a unique elliptic fibration.
\end{lemm}

\begin{proof}
Let us recall that a parabolic element of $\mathrm{SL}(2,\mathbb{Z})$ is up to conjugacy 
one of the following
\begin{align*}
&T_a^{+}=\left[\begin{array}{cc}
1 & a\\
0 & 1
\end{array}\right],&&
T_a^{-}=\left[\begin{array}{cc}
-1 & a\\
0 & -1
\end{array}\right], &&a\in\mathbb{Z}.
\end{align*}

\noindent For any $a\not=0$, the image $\theta(T_a^{+})$ of $T_a^{+}$ preserves a unique fibration on $\p^2$. Denote by~$\mathcal{F}$ the fibration preserved by $T_1^{+}$, given by $F\colon\mathbb{P}^2\dasharrow\mathbb{P}^1$. For any $a\not=0$, $T_a^{+}$ and $T_a^{-}$ commute with $T_{1}^{+}$ so the $\theta(T_a^{+})$'s and the $\theta(T_a^{-})$'s preserve the fibration $\mathcal{F}$ and $\mathcal{F}$ is the only fibration invariant by these elements. 

Let $M$ be a parabolic matrix. On the one hand $M$ is conjugate to $T_a^{+}$ or $T_a^{-}$ for some~$a$ via a matrix $N_M$ and on the other hand parabolic maps preserve a unique fibration; thus~$\theta(M)$ preserves the fibration given by $F\theta(N_M)^{-1}$. In particular if $F$ defines a rational (respectively elliptic) fibration, then $F\theta(N_M)^{-1}$ defines a rational (respectively elliptic) one.
\qed\end{proof}

\subsubsection{The standard embedding $\theta_s$}\label{Subsec:Standard}

The classical embedding
\begin{align*}
&\theta_s\colon\mathrm{SL}(2,\mathbb{Z})\to\mathrm{Bir}(\mathbb{P}^2),
&& \left[
\begin{array}{cc}
a & b\\
c & d
\end{array}
\right]\mapsto(x^ay^b,x^cy^d)
\end{align*}
preserves the type (\emph{see} for example \cite[Theorem 7.1]{JanLiLin}).

For any $M\in \mathrm{SL}(2,\mathbb{Z})$, if $M$ is elliptic, $\theta_s(M)$ is, up to conjugacy, one of the following birational maps of finite order 
\begin{align*}
&\left(\frac{1}{x},\frac{1}{y}\right), &&\left(y,\frac{1}{xy}\right),&&\left(y,
\frac{1}{x}\right), && \left(\frac{1}{y},
x\right), &&\left(\frac{1}{x},xy\right).
\end{align*}

\noindent If $M$ is parabolic, $\theta_s(M^n)$ is, up to conjugacy, $(xy^{na},y)$,  or  $(y^{na}/x,1/y)$ with~$a$ in~$\mathbb{Z}$ so $\theta_s(M)$ is parabolic. If $M$ is hyperbolic,~$M$ has two  real eigenvalues $\mu$ and $\mu^{-1}$ such that $\vert\mu\vert^{-1}<1<\vert\mu\vert$ and~$\lambda(\theta_s(M))=\vert\mu\vert>1$ and $\theta_s(M)$ is hyperbolic.

\bigskip

In \cite[page~9]{Fa}, a construction of a morphism $\mathrm{SL}(2,\mathbb{Z})\to\mathrm{Bir}(\mathbb{P}^2)$ preserving the type was given, inspired from \cite{CaLo} and \cite{Go}: the quotient of $\p^1\times \p^1$ by the involution $(x,y)\mapsto \left(\frac{1}{x},\frac{1}{y}\right)$ is a rational (singular) cubic surface $C\subset\p^3$, called \emph{Cayley cubic surface}. Explicitly, we can assume (by a good choice of coordinates) that $$C=\{(W:X:Y:Z)\in \p^3\ |\ XYZ+WYZ+WXZ+WXY=0\}$$ and that the quotient is given by
\begin{align*}
&\p^1\times\p^1\to  C, \\
&(x,y) \mapsto \left((x-1)(x-y)(1+y): (y-1)(y-x)(1+x):\vphantom{\Big(} (xy+1)(x+1)(y+1): (x-1)(y-1)(xy+1)\right).
\end{align*}

The involution $(x,y)\mapsto \left(\frac{1}{x},\frac{1}{y}\right)$ being the center of $\theta_s(\SL(2,\z))$, the quotient provides a morphism $\theta_s'\colon \SL(2,\z)\to \Bir(C)\simeq \Bir(\p^2)$ whose kernel is generated by $S^2$. The morphism preserves the type, but is not an embedding. It is also possible to deform the construction in order to have similar actions on other cubic surfaces (\emph{see} \cite{CaLo}).

\subsubsection{One first twisting of $\theta_s$}\label{ThetaMo}

We can "twist" the standard embedding $\theta_s$ in the following way.

Let $\theta_{-}(S)=\theta_s(S)=(y,\frac{1}{x})$ and $\theta_{-}(R)=(xy,-y)\not=\theta_s(R)=(xy,y)$. The map $\theta_{-}(RS)=\theta_{-}(R)\theta_{-}(S)=(\frac{y}{x},-\frac{1}{x})$ has order $3$. Since $\theta_{-}(R)$ commutes with $\theta_{-}(S^2)$, the relations of~$\mathrm{SL}(2,\z)$ are satisfied and $\theta_{-}$ is a morphism from $\mathrm{SL}(2,\z)$ into $\mathrm{Bir}(\p^2)$. 

%In fact, as we will see after, the map $\theta_{-}\colon \mathrm{SL}(2,\z)\to\Bir(\p^2)$ is an embedding that preserves the type and the groups $\theta_s(\mathrm{SL}(2,\mathbb{Z}))$ and $\theta_{-}(\mathrm{SL}(2,\mathbb{Z}))$ are not conjugate in the Cremona group.

 \begin{prop}\label{Prop:thetaMoins}
 The map $\theta_{-}\colon \mathrm{SL}(2,\z)\to\Bir(\p^2)$ is an embedding that preserves the type. 
 
The groups $\theta_s(\mathrm{SL}(2,\mathbb{Z}))$ and $\theta_{-}(\mathrm{SL}(2,\mathbb{Z}))$ are not conjugate in the Cremona group.
 \end{prop}
 
 \begin{proof}
 For each $M\in \mathrm{SL}(2,\z)$, one has $\theta_{-}(M)=\alpha_M\circ \theta_s(M)$ where $\alpha_M=(\pm x,\pm y)$, and in particular $\theta_{-}(M)$ and~$\theta_s(M)$ have the same degree. This observation implies that $\theta_{-}$ is an embedding, and that it preserves the type, since~$\theta_s $ does. 
 
 We now prove the second assertion. Suppose, for contradiction, that $\theta_s(\mathrm{SL}(2,\z))$ is conjugate to $\theta_{-}(\mathrm{SL}(2,\z))$; then~$\theta_s(R)=(xy,y)$ is conjugate to some parabolic element of  $\theta_{-}(\mathrm{SL}(2,\mathbb{Z}))$, which has no root in the group. This implies that $\theta_s(R)=(xy,y)$ or its inverse is conjugate to $\theta_{-}(R)=(xy,-y)$ or $\theta_{-}(RS^2)=\left(\frac{1}{xy},-\frac{1}{y}\right)$ in $\Bir(\p^2)$.
 
All these elements are parabolic elements of the Cremona group, each of them preserves a unique rational fibration, which is $(x,y)\mapsto y$. Since $\theta_s(R)$ preserves any fibre and both $\theta_{-}(R)$, $\theta_{-}(RS^2)$ permute the fibres, neither $\theta_s(R)$ nor $\theta_s(R^{-1})$ is  conjugate to  $\theta_{-}(R)$ or $\theta_{-}(RS^2)$ in~$\Bir(\p^2)$.
\qed\end{proof}

The map $\theta_{-}$ yields a "new" embedding of $\mathrm{SL}(2,\z)$ preserving the type. However, this map is not very far from the first one, and remains in $(\C^{*},\C^{*})\rtimes \mathrm{SL}(2,\z)$. We construct now new ones, more interesting. Conjugating the elements $\theta_{s}(S)=(y,\frac{1}{x})$ and $\theta_{s}(R)=(xy,y)$ by the birational map  $\left(\frac{x-1}{x+1},\frac{y-1}{y+1}\right)$, we get respectively  $(y,-x)$ and $\left(\frac{x+y}{xy+1},y\right)$.

More generally, we choose any $\varepsilon\in \C^{*}$, and set
\begin{align*}
&\theta_{\varepsilon}(S)=\left(y,-x\right), && \theta_{\varepsilon}(R)=\left(\frac{x+\varepsilon y}{\varepsilon+xy}, \varepsilon y\right).
\end{align*}
  
The map $\theta_{\varepsilon}(R)$ commutes with $\theta_{\varepsilon}(S^2)=\left(-x,-y\right)$, and $$\theta_{\varepsilon}(RS)=\left(\frac{y-\varepsilon x}{\varepsilon-xy}, -\varepsilon x\right)$$ is of order $3$, so $\theta_{\varepsilon}$ gives an homomorphism from $\mathrm{SL}(2,\z)$ to $\Bir(\p^2)$. The map $\theta_1$ being conjugate to the standard embedding, we can view this family as a deformation of the standard embedding. We prove now some technical results to show that the family consists of embedding preserving the type when $\varepsilon$ is a positive real number.

\begin{lemm}\label{Lem:BasePtsThetaEpsilon}
We view these maps on $\p^1\times \p^1$, via the embedding $(x,y)\mapsto ((x:1),(y:1))$. 
  
$(i)$ Writing $R_1=\left[\begin{array}{cc}
1 & 1\\
0 & 1
\end{array}\right] $, $R_2=\left[\begin{array}{cc}
1 & 0\\
1 & 1
\end{array}\right] $, 
both maps 
\begin{align*}
&\theta_{\varepsilon}(R_1)=\left(\frac{x+\varepsilon y}{\varepsilon+xy}, \varepsilon y\right)&& \text{and} &&\theta_{\varepsilon}(R_2)=\left(\frac{x}{\varepsilon}, \frac{\varepsilon(x+\varepsilon y)}{\varepsilon+xy}\right)
\end{align*}
have exactly two base-points both belonging to $\p^1\times \p^1$ $($no infinitely near point$)$, and being $p_1=(\varepsilon,-1)$ and $p_2=(-\varepsilon,1)$ $($or $((\varepsilon:1),(-1:1))$ and $((-\varepsilon:1),(1:1)))$.

$(ii)$ Both maps 
\begin{align*}
&\theta_{\varepsilon}(R_1)^{-1}=\left(\frac{\varepsilon(\varepsilon x- y)}{\varepsilon-xy}, \frac{y}{\varepsilon}\right)&& \text{and} &&\theta_{\varepsilon}(R_2)^{-1}=\left(\varepsilon x, \frac{y-\varepsilon x}{\varepsilon-xy}\right)
\end{align*}
have exactly two base-points, being $q_1=(1,\varepsilon)$ and $q_2=(-1,-\varepsilon)$.

$(iii)$ 
If $\varepsilon$ is a positive real number and  $M=R_{i_k}\dots R_{i_1}$, for $i_1,\dots,i_k\in \{1,2\}$, the follo\-wing hold:

\begin{itemize}
\item the points $q_1$ and $q_2$ are not base-points of $\theta_\varepsilon(M)$, and $\theta_\varepsilon(M)(\{q_1,q_2\})\cap \{p_1,p_2\}=~\emptyset$.

\item  the points $p_1$ and $p_2$ are not base-points of $\theta_\varepsilon(M^{-1})$, and $\theta_\varepsilon(M^{-1})(\{p_1,p_2\})\cap \{q_1,q_2\}=\emptyset$.
\end{itemize}
\end{lemm}

\begin{proof}
Parts $(i)$ and $(ii)$ follow from an easy calculation, it remains to prove $(iii)$.

Let $U_{+}\subset \mathbb{R}^2\subset\p^1\times \p^1$  (resp. $U_{-}\subset \mathbb{R}^2\subset \p^1\times\p^1$) be the subset of points  $(x,y)$ with~$x$, $y\in \mathbb{R}$,  $xy>0$ (resp. $xy<0$). When $\varepsilon$ is a positive real number,  $\{p_1,p_2\}\subset U_-$ and~$\{q_1,q_2\}\subset U_+$, which implies that $\theta_\varepsilon(R_i)$ (resp. $\theta_\varepsilon(R_i^{-1})$)  is defined at any point of~$U_{+}$ (resp. of $U_-$), since $U_+\cap U_-=\emptyset$.

Moreover, the explicit form of the four maps given in $(i)$, $(ii)$ shows that $\theta_\varepsilon(R_i)(U_+)\subset~U_+$ and $\theta_\varepsilon(R_i^{-1})(U_-)\subset~U_-$  for $i=1,2$. This yields the result.
\qed\end{proof}

Recall that $\Pic{\p^1\times \p^1}=\z f_1\oplus \z f_2$, where $f_i$ is the fibre of the projection on the~$i$-th factor. In particular, any curve on $\p^1\times \p^1$ has a bidegree~$(d_1,d_2)$ and any element of~$\mathrm{Bir}(\p^1\times\p^1)$ has a quadridegree, which is given by the two bidegrees of the pull-backs of~$f_1$ and $f_2$, or equivalently by the two bidegrees of the polynomials which define the map.

Remark that the dynamical degree of a birational map  $\varphi$ of $\p^1\times\p^1$ is uniquely determined by the sequence of quadridegrees of $\varphi^n$.

\begin{prop}
If $\varepsilon$ is a positive real number, the following hold:

$(i)$ For any $M=\left[\begin{array}{cc}
a & b\\
c & d
\end{array}\right] \in \mathrm{SL}(2,\z)$, the maps $\theta_\varepsilon(M)$ and $\theta_s(M)$ have the same quadridegree as birational maps of $\p^1\times\p^1$, which is $(\vert a\vert,\vert b\vert,\vert c\vert,\vert d\vert)$.

$(ii)$  The homomorphism $\theta_\varepsilon$ is an embedding of $\mathrm{SL}(2,\z)$ into the Cremona group that preserves the type. 
\end{prop}

\begin{proof}
Observe first that $(i)$ implies that the kernel of $\theta_\varepsilon$ is trivial (since $\theta_\varepsilon(S^2)=(-x,-y)$ is not trivial) so that $\theta_\varepsilon$ is an embedding, and also implies that the dynamical degree of~$\theta_\varepsilon(M)$ and $\theta_s(M)$ are the same for any $M$. This shows that $(i)$ implies~$(ii)$.

We now prove assertion $(i)$. 
Since $\theta_s(S)=\left(y,\frac{1}{x}\right)$ and $\theta_\varepsilon(S)=(y,-x)$ are automorphisms of $\p^1\times\p^1$ having the same action on $\Pic{\p^1 \times \p^1}$, $\theta_\varepsilon(M)$ and $\theta_s(M)$ have the same quadridegree if and only if $\theta_\varepsilon(MS)$ and $\theta_s(MS)$ have the same quadridegree. The same holds when we multiply on the left:  $\theta_\varepsilon(M)$ and $\theta_s(M)$ have the same quadridegree if and only if $\theta_\varepsilon(SM)$ and $\theta_s(SM)$ have the same quadridegree.

Recall that $\mathrm{SL}(2,\z)$ has the presentation $\langle R,\,RS\,\vert\, S^4=(RS)^3=1,\,S^2(RS)=(RS)S^2\rangle$. It suffices thus to prove that~$\theta_\varepsilon(M)$ and $\theta_s(M)$ have the same quadridegree when $M=(RS)^{i_k}\dots S(RS)^{i_2}S(RS)^{i_1}S$, for some $i_1,\dots,i_k\in \{\pm 1\}$. For any index $i_j$ equal to $1$, we replace the $S$ immediately after by $S^{-1}$ (since $S^2$ commutes with all matrices), and obtain now a product of non-negative powers of $(RS)S^{-1}=R$ and $(RS)^2S$. We will write $R_1=R$ and $R_2=(RS)^2S$, and have $$R_1=\left[\begin{array}{cc}
1 & 1\\
0 & 1
\end{array}\right] , \ \ R_2=\left[\begin{array}{cc}
1 & 0\\
1 & 1
\end{array}\right].$$

It is thus sufficient to prove the following assertion:

\begin{center}\it $(\star)$ if $M=\left[\begin{array}{cc}
a & b\\
c & d
\end{array}\right]=R_{i_k}R_{i_{k-1}}\dots R_{i_1}$, for some $i_1,\dots,i_k\in \{1,2\}$,\\  then $a,b,c,d\geq 0$, and $\theta_s(M)$, $\theta_\varepsilon(M)$ have both quadridegree~$(a,b,c,d)$.\end{center}

We proceed now by induction on $k$. For $k=1$, Assertion~$(\star)$ can be directly checked:

Both $\theta_s(R_1)=(xy,y)$ and $\theta_{\varepsilon}(R_1)=\left(\frac{x+\varepsilon y}{\varepsilon+xy}, \varepsilon y\right)$ have quadridegree~$(1,1,0,1)$. Both $\theta_s(R_2)=(x,xy)$ and $\theta_{\varepsilon}(R_2)=\left(\frac{x}{\varepsilon}, \frac{\varepsilon(x+\varepsilon y)}{\varepsilon+xy}\right)$ have quadridegree~$(1,0,1,1)$.

Now, assume that $(\star)$ is true for $M=\left[\begin{array}{cc}
a & b\\
c & d
\end{array}\right]$, and let us prove it for $R_1M=\left[\begin{array}{cc}
a+c & b+d\\
c & d
\end{array}\right]$ and $R_2M=\left[\begin{array}{cc}
a & b\\
a+c & b+d
\end{array}\right]$. By induction hypothesis one has $$
\theta_\varepsilon(M)=((x_1:x_2),(y_1:y_2))\dasharrow ((P_{1}:P_2),(P_3:P_4)),$$ where $P_1$, $P_2$, $P_3$, $P_4\in \C[x_1,x_2,y_1,y_2]$ are bihomogeneous polynomials, of bidegree~$(a,b)$, $(a,b)$, $(c,d)$, $(c,d)$.

We have thus 
$$\begin{array}{rcl}
\theta_\varepsilon(R_1)\theta_\varepsilon(M)=\theta_\varepsilon(R_1M)=\\
((x_1:x_2),(y_1:y_2))&\dasharrow &((P_{1}P_4+\varepsilon P_2P_3:\varepsilon P_2P_4 +P_1P_3),(\varepsilon P_3:P_4)),\vspace{0.1cm}\\
\theta_\varepsilon(R_2)\theta_\varepsilon(M)=\theta_\varepsilon(R_2M)=\\
((x_1:x_2),(y_1:y_2))&\dasharrow &((P_1:\varepsilon P_2),(\varepsilon(P_{1}P_4+\varepsilon P_2P_3):\varepsilon P_2P_4 +P_1P_3)).\end{array}$$

To prove $(\star)$ for $R_1M$ and $R_2M$, it suffices to show that the polynomials $P_{1}P_4+\varepsilon P_2P_3$ and $\varepsilon P_2P_4 +P_1P_3$ have  no common component.
Suppose the converse for contradiction, and denote by $h\in \C[x_1,x_2,y_1,y_2]$ the common component. The polynomial $h$ corresponds to a curve of $\p^1\times\p^1$ that is contracted by $\theta_\varepsilon(M)$ onto a base-point of $\theta_\varepsilon(R_1)$ or~$\theta_\varepsilon(R_2)$, \emph{ i.e.}\ onto $p_1=(\varepsilon,-1)$ or $p_2=(-\varepsilon,1)$ (Lemma~\ref{Lem:BasePtsThetaEpsilon}). 
But this condition means that $(\theta_\varepsilon(M))^{-1}$ has a base-point at $p_1$ or~$p_2$. We proved in Lemma~\ref{Lem:BasePtsThetaEpsilon} that this is impossible when~$\varepsilon$ is a positive real number.
\qed\end{proof}

We now show that this construction yields infinitely many conjugacy classes of embeddings of $\mathrm{SL}(2,\mathbb{Z})$ into the Cremona group that preserve the type.

 \begin{prop}\label{ProP:epsilonespilonPrime}
If $\varepsilon$ and $\varepsilon'$ are two real positive numbers with $\varepsilon\varepsilon'\not=1$, the two groups $\theta_{\varepsilon}(\mathrm{SL}(2,\mathbb{Z}))$ and $\theta_{\varepsilon'}(\mathrm{SL}(2,\mathbb{Z}))$ are not conjugate in the Cremona group.

The standard embedding $\theta_s$ is conjugate to $\theta_1$, but $\theta_{-}(\mathrm{SL}(2,\mathbb{Z}))$ is not conjugate to $\theta_{\varepsilon}(\mathrm{SL}(2,\mathbb{Z}))$ for any  positive $\varepsilon\in \mathbb{R}$.
 \end{prop}
 
 \begin{proof}
 The proof is similar to the one of Proposition~\ref{Prop:thetaMoins}. Assume, for contradiction, that~$\theta_\varepsilon(\mathrm{SL}(2,\z))$ is conjugate to $\theta_{\varepsilon'}(\mathrm{SL}(2,\z))$; then $\theta_\varepsilon(R)=\left(\frac{x+\varepsilon y}{\varepsilon+xy}, \varepsilon y\right)$ is conjugate to some parabolic element of  $\theta_{\varepsilon'}(\mathrm{SL}(2,\mathbb{Z}))$, which has no root in the group. This implies that~$\theta_\varepsilon(R)=\left(\frac{x+\varepsilon y}{\varepsilon+xy}, \varepsilon y\right)$ or its inverse is conjugate to $\theta_{\varepsilon'}(R)=\left(\frac{x+\varepsilon' y}{\varepsilon'+xy}, \varepsilon' y\right)$ or to  $\theta_{\varepsilon'}(RS^2)=\left(\frac{-x-\varepsilon' y}{\varepsilon'+xy}, -\varepsilon' y\right)$ in $\Bir(\p^2)$.
 
These  elements are parabolic elements of the Cremona group, each of them preserves a unique rational fibration, which is $(x,y)\mapsto y$. The action on the basis being different up to conjugacy (since $\varepsilon\varepsilon'\not=\pm 1$), neither $\theta_{\varepsilon}(R)$ nor its inverse is  conjugate to  $\theta_{\varepsilon'}(R)$ or $\theta_{\varepsilon'}(RS^2)$ in~$\Bir(\p^2)$.

It remains to show that $\theta_{-}(\mathrm{SL}(2,\mathbb{Z}))$ is not conjugate to $\theta_{\varepsilon}(\mathrm{SL}(2,\mathbb{Z}))$ for any positive $\varepsilon\in \mathbb{R}$.
Every parabolic element of $\theta_{-}(\mathrm{SL}(2,\mathbb{Z}))$ without root is conjugate to $\theta_{-}(R)=(xy,-y)$, $\theta_{-}(RS^2)=(\frac{1}{xy},-\frac{1}{y})$ or their inverses, and acts thus non-trivially on the basis of the unique fibration preserved, with an action of order $2$. We get the result by observing that $\theta_{\varepsilon}(\mathrm{SL}(2,\mathbb{Z}))$ contains $\theta_\varepsilon(R)=\left(\frac{x+\varepsilon y}{\varepsilon+xy}, \varepsilon y\right)$, which is parabolic, without root and acting on the basis with an action which has not order $2$.
\qed\end{proof}

Note that in all our examples of embeddings preserving the type, the parabolic elements have a linear degree growth. One can then ask the following question (which could yield a positive answer to Question~\ref{Ques:OtherEmbedd}).
\begin{ques}
Does there exist an embedding of ~$\SL(2,\z)$ into $\Bir(\p^2)$ that preserves the type and such that the degree growth of parabolic elements is quadratic?\end{ques}
\subsection{Elliptic embeddings}\label{Sec:ReprTypEllipt}
The simplest elliptic embedding  is given by
\begin{align*}
&\theta_e\colon\mathrm{SL}(2,\mathbb{Z})\to\mathrm{Bir}(\mathbb{P}^2), &&\left[\begin{array}{cc}
a & b \\
c & d
\end{array}
\right]\mapsto (ax+by:cx+dy:z).
\end{align*}

We now generalise this embedding. Choose $n\in\mathbb{N}$ and let $\chi\colon \mathrm{SL}(2,\mathbb{Z})\to \mathbb{C}^*$ be a character such that $\chi\left(\left[\begin{array}{cc}
-1 & 0 \\
0 & -1
\end{array}
\right]\right)\not=(-1)^n$. For simplicity, we choose $\chi$ such that~$\chi(RS)=~1$, and such that $\chi(S)$ is equal to $1$ if $n$  is odd and to $\im$ if  $n$ is even. Then we define $\theta_{n}\colon \mathrm{SL}(2,\mathbb{Z})\to \Bir(\p^2)$ by

$$M=\left[\begin{array}{cc}
a & b\\
c & d
\end{array}\right]\mapsto \left(\frac{ax+b}{cx+d},\frac{\chi(M)y}{(cx+d)^n}\right).$$

The action on the first component and the fact that $\theta_n(S^2)\not=1$ imply that $\theta_n$ is an embedding. The degree of all elements being bounded, the embeddings are elliptic.

 \begin{prop}\label{ellipticemb}
 For any $n\in \mathbb{N}$,  the group $\theta_n(\mathrm{SL}(2,\mathbb{Z}))$ is conjugate to a subgroup of~$\Aut(\mathbb{F}_n)$, where $\mathbb{F}_n$ is the $n$-th Hirzebruch surface.
 
The groups $\theta_m(\mathrm{SL}(2,\mathbb{Z}))$ and $\theta_n(\mathrm{SL}(2,\mathbb{Z}))$ are conjugate in the Cremona group if and only if $m=n$.
 \end{prop}
 \begin{proof}
 If $n=0$, the embedding $(x,y)\mapsto ((x:1),(y:1))$ of $\C^2$ into $\p^1\times\p^1=\mathbb{F}_0$ conjuga\-tes~$\theta_0(\mathrm{SL}(2,\mathbb{Z}))$ to a subgroup of $\Aut(\mathbb{F}_0)$.

For $n\ge 1$, recall that the weighted projective space $\p(1,1,n)$ is equal to $$\p(1,1,n)=\left\{(x_1,x_2,z)\in \C^3\backslash \{0\}\ \Big|\ (x_1,x_2,z)\sim (\mu x_1,\mu x_2,\mu^n z),\ \mu\in \C^*\right\}.$$
 
 The surface $\p(1,1,1)$ is equal to $\p^2$, and the surfaces $\p(1,1,n)$ for $n\ge 2$ have one singular point, which is $(0:0:1)$.
 
 For any $n\ge 1$, the embedding $(x,y)\mapsto (x:y:1)$ of $\C^2$ into $\p(1,1,n)$ conjuga\-tes~$\theta_n(\mathrm{SL}(2,\mathbb{Z}))$ to a subgroup of~$\Aut(\p(1,1,n))$ that fixes the point $(0:0:1)$. The blow-up of this fixed point gives the Hirzebruch surface $\mathbb{F}_n$, and conjugates thus $\theta_n(\mathrm{SL}(2,\mathbb{Z}))$ to a subgroup of~$\Aut(\mathbb{F}_n)$.

 In all cases $n\ge 0$, the group preserves the fibration $\mathbb{F}_n\to \p^1$ corresponding to $(x,y)\mapsto~x$. The action on the basis of the fibration corresponds to the standard homomorphism $\SL(2,\z)\to \mathrm{PSL}(2,\z)\subset \PGL(2,\mathbb{C})=\Aut(\p^1)$. This action has no orbit of finite size on $\p^1$. In particular, there is no orbit of finite size on $\mathbb{F}_n$. This shows that the subgroup of $\Aut(\mathbb{F}_n)$ corresponding to $\theta_n(\SL(2,\z))$ is birationally rigid for $n\not=1$, \emph{i.e.\ }that it is not conjugate to any group of automorphisms of any other smooth projective surface. This shows  that~$\theta_m(\mathrm{SL}(2,\mathbb{Z}))$ and $\theta_n(\mathrm{SL}(2,\mathbb{Z}))$ are conjugate in the Cremona group only when $m=~n$.
 \qed\end{proof}
 
\subsection{Parabolic embeddings}\label{Sec:ReprTypParab}
Recall that the morphism $\theta_0$ defined in \S \ref{Sec:ReprTypEllipt} can also be viewed as follow: $M=\left[\begin{array}{cc}
a & b\\
c & d
\end{array}\right]\mapsto \left(\frac{ax+b}{cx+d},\chi(M)y\right)$;
it preserves the fibration $(x,y)\mapsto x$. Remembering that  $\chi(S)=\im$ and~$\chi(RS)=~1$ we have 

\begin{align*}
&\theta_0(S)=\left(-\frac{1}{x},\im y\right) && \text{and} &&\theta_0(RS)=\left(\frac{x-1}{x},y\right).
\end{align*}

We will ``twist'' $\theta_0$ in order to construct parabolic embeddings. Recall that~$\mathrm{SL}(2,\mathbb{Z})$ acts via $\theta_0$ on the projective line; the element $\left[\begin{array}{cc}
a & b\\
c & d
\end{array}\right]$ acts as $x\dasharrow \frac{ax+b}{cx+d}$. The group is countable so a very general point of the line has no isotropy. Let $P\in \mathbb{C}(x)$ be a rational function with $m$ simple poles and $m$ simple zeroes, where~$m>0$, and such that the $2m$ corresponding points of $\mathbb{C}$ are all on different orbits under the action of $\mathrm{SL}(2,\mathbb{Z})$ and have no isotropy. We denote by $\varphi_P=(x,y\cdot P(x))$ the associated birational map; it preserves the fibration and commutes with $\theta_0(S^2)=(x,-y)$.

We choose
\begin{align*}
\theta_P(S)=\theta_0(S)=\left(-\frac{1}{x},\im y\right)&& \text{and} &&\theta_P(RS)=\varphi_P\circ  \theta_0(RS)\circ \varphi_P^{-1},
\end{align*}
therefore
\begin{align*}
&\theta_P(S)=\left(-\frac{1}{x},\im y\right) && \text{and} &&\theta_P(RS)=\left(\frac{x-1}{x},y\cdot \frac{P(\frac{x-1}{x})}{P(x)}\right).
\end{align*}

The maps $\varphi_P$ and $\theta_P(S^2)$ commute so $\theta_P(RS)$ and $\theta_P(S^2)$ commute too. Then, by definition of~$\theta_P(S)$ and~$\theta_P(RS)$ there is a unique morphism $\theta_P\colon \mathrm{SL}(2,\mathbb{Z})\to \mathrm{Bir}(\mathbb{C}^2)$. 

\begin{prop}\label{parabolicemb1}
The morphism $\theta_P$ is a parabolic embedding for any $P\in \mathbb{C}(x)$.
\end{prop}

\begin{proof}
The action on the basis of the fibration and the fact that $\theta_P(S^2)\not=\mathrm{id}$ imply that~$\theta_P$ is an embedding. It remains to show that any element of infinite order is sent onto a parabolic element.

Writing $\alpha=\theta_P(RS)$ and $\beta=\theta_P(S)$, it suffices to show that $h$ or $h\beta^2$ is parabolic, where
\begin{align*}
&h=\beta\alpha^{i_n}\beta\dots \alpha^{i_2}\beta \alpha^{i_1},&& n\geq 1 &&\text{ and } &&i_1,\dots,i_n\in \{-1,1\}.
\end{align*}
We view our maps acting on $\p^1\times \p^1$. The fibration given by the projection on the first factor is preserved by $h$, which is thus either parabolic or elliptic. The first possibility occurs if the sequence of number of base-points of $h^k$ grows linearly and the second if the sequence is bounded.

Let $p\in \C$ be a pole or a zero of $P$. Let $F_0\subset \mathbb{P}^1\times\mathbb{P}^1$ be the fibre of $(p:1)$ and let~$\Sigma\subset \mathbb{P}^1\times\p^1$ be the $($countable$)$ union of fibres of points that belong to the orbit of~$(p:1)$ under the action of $\mathrm{SL}(2,\mathbb{Z})$.

Recall that $\theta_0(RS)$ is an automorphism of $\mathbb{P}^1\times\mathbb{P}^1$. Set $F_1=\theta_0(RS)(F_0)$ and $F_2=\theta_0(RS)(F_1)$; remark that $F_0=\theta_0(RS)(F_2)$. Then $\varphi_P$ and its inverse contract $F_0$ on a point of $F_0$ but send isomorphically $F_1$ and~$F_2$ onto themselves. The map $\alpha$ is the conjugate of $ \theta_0(RS)$ by $\varphi_P$, so it contracts~$F_0$ and $F_2$ on points lying respectively on~$F_1$ and $F_0$, but sends isomorphically $F_1$ onto $F_2$ and doesn't contract any other fibre contained in $\Sigma$. Similarly~$\alpha^{-1}$ contracts $F_0$ and~$F_1$ on points lying on $F_2$ and $F_0$ and neither contracts $F_2$ nor any other fibre of~$\Sigma$.

Each fibre is preserved by $\beta^2$, but $\beta$ and $\beta^3$ send $F_0$, $F_1$, $F_2$ onto three other fibres contained in~$\Sigma$. Then~$\alpha^{\pm 1}\beta$ and~$\alpha^{\pm 1}\beta^3$ send isomorphically $F_0$ onto a fibre contained in~$\Sigma\setminus\{F_i\}$. By induction on $n$, we obtain that for any $k<0$, $h^k$ and $(h\beta^2)^k$ send isomorphically~$F_0$ onto a curve in $\Sigma\setminus\{F_i\}$. 

Then we note that $\alpha$ and $\alpha^{-1}$ contract $F_0$ on a point contained in one of the $F_i$, point sent by~$\beta$ onto an other point not contained in the $F_i$'s. So, by induction on $n$, for any $k>0$ both $h^k$ and $(h\beta^2)^k$ contract~$F_0$ on a point not contained in the $F_i$'s and for which the fibre belongs to $\Sigma$.

\smallskip

For each integer $k>0$, the fibre $F_0$ is contracted by~$h^k$ and by~$(h\beta^2)^k=h^k(\beta^{2k})$ on a point of $\Sigma$. Moreover, for each integer $k<0$, $F_0$ is sent isomorphically by~$h^{k}$ onto a fibre contained in $\Sigma$. Set $F_i'=h^{-i}(F_0)$ for all $i>0$; we obtain that~$h^k$ and $(h\beta^2)^k$ contract~$F_0$ and~$F_1',\ldots, F_k'$ for each inte\-ger~$k>0$.  This means that the number of base-points of $h^k$ and~$(h\beta^2)^k$ is at least equal to $k$.  As~$h$ and~$h\beta^2$ preserve the fibration, they are parabolic. 
\qed\end{proof}

\begin{prop}\label{parabolicemb2}
When $P$ varies, we obtain infinitely many parabolic embeddings.
\end{prop}

 \begin{proof}Let $P,Q\in \C(x)$, and suppose that $\theta_P(\mathrm{SL}(2,\z))$ is conjugate to $\theta_Q(\mathrm{SL}(2,\z))$ by some birational map $\varphi$ of~$\p^1\times \p^1$. Then $\varphi$ preserves the fibration $(x,y)\mapsto x$, which is the unique fibration preserved by the two groups. Its action on the basis of the fibration is an element $\psi\in \PGL(2,\C)$  that normalises $\mathrm{PSL}(2,\z)\subset \mathrm{PSL}(2,\C)=\PGL(2,\C)$. This means that $\psi\in \mathrm{PSL}(2,\z)$.
 Replacing $\varphi$ by its product with an element of $\theta_Q(\mathrm{SL}(2,\z))$, we can thus assume that~$\varphi$ acts trivially on the basis.
 
This means that  $\varphi$ is equal to $\left(x,\frac{a(x)y+b(x)}{c(x)y+d(x)}\right)$ for some $a,b,c,d\in \C(x)$, $ad-bc\not=0$. Since $\varphi$ conjugates $\theta_P(S)=\theta_Q(S)=(-\frac{1}{x},\im y)$ to itself or its inverse, the map $\varphi$ is equal to~$(x,a(x)y^{\pm 1})$ where $a\in \C(x)$, $a(-\frac{1}{x})=\pm a(x)$.

The map $\varphi$ conjugates $\theta_P(RS)=\left(\frac{x-1}{x},y\cdot \frac{P(\frac{x-1}{x})}{P(x)}\right)$  to $\theta_Q(RS)=\left(\frac{x-1}{x},y\cdot \frac{Q(\frac{x-1}{x})}{Q(x)}\right)$ or to~$\theta_Q(RS^3)=\left(\frac{x-1}{x},-y\cdot \frac{Q(\frac{x-1}{x})}{Q(x)}\right)$ in $\Bir(\p^1\times \p^1)$. Assume that 
\begin{align*}
&\varphi=(x,a(x)y) && \text{where }a\in \C(x),\, a(-\frac{1}{x})=a(x);
\end{align*}
then $\varphi\theta_P(RS)\varphi^{-1}=\left(\frac{x-1}{x},y\cdot\frac{a\left(\frac{x-1}{x}\right)P\left(\frac{x-1}{x}\right)}{a(x)P(x)}\right)$. Thus $\varphi\theta_P(RS)\varphi^{-1}=\theta_Q(RS),$ resp. $\theta_Q(RS^3)$ if and only if 
\begin{align*}
&\frac{a\left(\frac{x-1}{x}\right)}{a(x)}=\frac{P(x)Q(\frac{x-1}{x})}{Q(x)P\left(\frac{x-1}{x}\right)}, && \text{resp. }\frac{a\left(\frac{x-1}{x}\right)}{a(x)}=-\frac{P(x)Q(\frac{x-1}{x})}{Q(x)P\left(\frac{x-1}{x}\right)}
\end{align*}
since $a(x)$ is invariant under the homography $x\mapsto -\frac{1}{x}$, the same holds for $\frac{P(x)Q(\frac{x-1}{x})}{Q(x)P\left(\frac{x-1}{x}\right)}$. This implies, in both cases, the following condition on $P$ and $Q$
$$\frac{P(x)P(1+x)}{P\left(-\frac{1}{x}\right)P\left(\frac{x-1}{x}\right)}=\frac{Q(x)Q(1+x)}{Q\left(-\frac{1}{x}\right)Q\left(\frac{x-1}{x}\right)}.$$
We get the same formula when $\varphi$ is equal to $(x,a(x)y^{-1})$ where $a\in \C(x)$, $a(-\frac{1}{x})=-a(x)$.
When $P$ varies, we thus obtain infinitely many parabolic embeddings.
\qed\end{proof}

\subsection{Hyperbolic embeddings}\label{Sec:ReprTypHyperZ}

In this section, we "twist" the standard elliptic embedding $\theta_e$ defined in $\S\ref{Sec:ReprTypEllipt}$ to get many hyperbolic embeddings of $\SL(2,\z)$ into $\Bir(\p^2)$. Recall that $\theta_e$ is given by 
\begin{align*}
&\theta_e\colon\mathrm{SL}(2,\mathbb{Z})\to\mathrm{Bir}(\mathbb{P}^2), &&\left[\begin{array}{cc}
a & b \\
c & d
\end{array}
\right]\mapsto (ax+by:cx+dy:z).
\end{align*}
The group $\theta_e(\SL(2,\z))$ preserves the line $L_z$ of equation $z=0$, and acts on it via the natural maps $\SL(2,\z)\to \PSL(2,\z)\subset \PSL(2,\C)=\Aut(L_z)$.

We choose $\mu \in \C^{*}$ such that the point $p=(\mu:1:0)\in L_z$ has a trivial isotropy group under the action of $\PSL(2,\z)$,  fix an even integer $k>0$, and then define  a morphism $\theta_{k}\colon \SL(2,\z)\to \Bir(\p^2)$ by the following way:

$$\begin{array}{rcl}\theta_{k}(S)&=&\theta_e(S)=(y:-x:z)\\
\theta_{k}(RS)&=&\psi \theta_e(RS) \psi^{-1}\end{array}$$
where $\psi$ is the conjugation of $\psi'=(x^k:yx^{k-1}+z^k:zx^{k-1})$ by $(x+\mu y:y:z)$. 

Note that $\psi'$ restricts to an automorphism of the affine plane where $x\not=0$,  commutes with $\theta_e(S^2)=(x:y:-z)$ and acts trivially on $L_z$.
Since $\psi$ commutes with $\theta_e(S^2)=\theta_{k}(S^2)$, the element $\theta_{k}(RS)$ commutes with $\theta_{k}(S^2)$, and $\theta_{k}$ is thus a well-defined morphism. The fact that $\psi$ preserves $L_z$ and acts trivially on it implies that the action of $\theta_e$ and $\theta_{k}$ on $L_z$ are the same, so $\theta_{k}$ is an embedding.

\begin{lemm}\label{Lem:HypDeg}Let $m$ be a positive integer, and let $a_1,\dots,a_m,b_1,\dots,b_m\in \{\pm 1\}$.
The birational map
$$\theta_k(S^{b_m}(RS)^{a_m}\cdots S^{b_1}(RS)^{a_1})$$ has degree $k^{2m}$ and exactly $2m$ proper base-points, all lying on $L_z$, which are 
$$\begin{array}{l}
p, ((RS)^{a_1})^{-1}(p),(S^{b_1}(RS)^{a_1})^{-1}(p),((RS)^{a_2}S^{b_1}(RS)^{a_1})^{-1}(p),\\
\dots,((RS)^{a_m}\cdots S^{b_1}(RS)^{a_1})^{-1}(p),(S^{b_m}(RS)^{a_m}\cdots S^{b_1}(RS)^{a_1})^{-1}(p),\end{array}$$
where the action of $R,RS\in \SL(2,\z)$ on $L_z$ is here the action via $\theta_e$ or $\theta_k$.
\end{lemm}
\begin{proof}
The birational map $\psi$ has degree $k$ and has an unique proper base-point which is $p=(\mu:1:0)\in L_z$; the same is true for $\psi^{-1}$. Moreover both maps fix any other point of $L_z$. 

Since $\theta_e(RS)^{a_1}$ is an automorphism of $\p^2$ that moves the point $p$ onto an other point of $L_z$, the map $\theta_k( (RS)^{a_1})=\psi \theta_e(RS)^{a_1} \psi^{-1}$ has degree $k^2$ and exactly two proper base-points, which are $p$ and $\psi\theta_e(RS)^{-a_1}(p)=((RS)^{a_1})^{-1}(p)$. The map $\theta_k(S)$ being an automorphism of $\p^2$, $\theta_k(S^{b_1}(RS)^{a_1})$ has also degree $k^2$ and two proper base-points, which are $p$ and $((RS)^{a_1})^{-1}(p)$. This gives the result for $m=1$.

Proceeding by induction for $m>1$, we assume that $\theta_k(S^{b_m}(RS)^{a_m}\cdots S^{b_2}(RS)^{a_2})$ has degree $k^{2m-2}$ and exactly $2m-2$ proper base-points, all lying on $L_z$, which are 
$$
p, ((RS)^{a_2})^{-1}(p),(S^{b_2}(RS)^{a_2})^{-1}(p),\dots,(S^{b_m}(RS)^{a_m}\cdots S^{b_2}(RS)^{a_2})^{-1}(p).$$
The map $\theta_k(S^{b_1}(RS)^{a_1})^{-1}=\theta_k((RS)^{-a_1})\theta_k(S^{-b_1})$ has degree $k^2$ and two proper base-points, which are $S^{b_1}(p)$ and $S^{b_1}(RS)^{a_1}(p)$. These two points being distinct from the $2m-2$ points above, the map   $\theta_k(S^{b_m}(RS)^{a_m}\cdots S^{b_1}(RS)^{a_1})$ has degree $k^2\cdot k^{2m-2}=k^{2m}$, and its proper base-points are the $2$ proper base-points of $\theta_k(S^{b_1}(RS)^{a_1})$ and the image by $(S^{b_1}(RS)^{a_1})^{-1}$ of the base-points of  $\theta_k(S^{b_m}(RS)^{a_m}\cdots S^{b_2}(RS)^{a_2})$. This gives the result.\qed
\end{proof}
As a corollary, we get infinitely many hyperbolic embeddings of $\SL(2,\z)$ into the Cremona group.
\begin{coro}\label{CoroHyp}
Let $m$ be a positive integer, and let $a_1,\dots,a_m,b_1,\dots,b_m\in \{\pm 1\}$.
The birational map
$$\theta_k(S^{b_m}(RS)^{a_m}\cdots S^{b_1}(RS)^{a_1})$$ has dynamical degree $k^{2m}$.

In particular, the map $\theta_k$ is an hyperbolic embedding and the set of all dynamical degrees of $\theta_k(\SL(2,\z))$ is $\{1,k^2,k^4,k^6,\dots\}$. 
\end{coro}
\begin{proof}
Any element of infinite order of $\SL(2,\z)$ is conjugate to $g=S^{b_m}(RS)^{a_m}\cdots S^{b_1}(RS)^{a_1}$ for some $a_1,\dots,a_m,b_1,\dots,b_m\in \{\pm 1\}$. Lemma~\ref{Lem:HypDeg} implies that the degree of $\theta_k(g^r)$ is equal to $k^{2mr}$. The dynamical degree of $\theta_k(g)$ is therefore equal to $k^{2m}$. \qed
\end{proof}

\section{Description of hyperbolic embeddings for which the central element fixes (pointwise) an elliptic curve}\label{Sec:invfixbourbell}
\subsection{Outline of the construction and notation}
In this section, we give a general way of constructing embeddings of $\mathrm{SL}(2,\mathbb{Z})$ into the Cremona group where the central involution fixes pointwise an elliptic curve. Recall that all conjugacy classes of elements of order $4$ or $6$ in $\mathrm{Bir}(\mathbb{P}^2)$ have been classified (\emph{see} \cite{BlaC}). Many of them can act on del Pezzo surfaces of degree~$1$, $2$, $3$ or $4$. 

In order to create our embedding, we will define del Pezzo surfaces $X$, $Y$ of degree~$\le 4$, and automorphisms $\alpha\in \Aut(X)$, $\beta\in \Aut(Y)$ of order respectively $6$ and $4$, so that $\alpha^3$ and~$\beta^2$ fix pointwise an elliptic curve, and that $\Pic{X}^{\alpha}$, $\Pic{Y}^{\beta}$ have both rank $1$. Note that we say that a curve is \emph{fixed} by a birational map if it is pointwise fixed, and say that it is \emph{invariant} or \emph{preserved} if the map induces a birational action (trivial or not) on the curve. Contracting $(-1)$-curves invariant by these involutions (but not by $\alpha$, $\beta$, which act minimally on $X$ and $Y$), we obtain birational morphisms $X\to X_4$ and $Y\to Y_4$, where $X_4$, $Y_4$ are del Pezzo surfaces on which $\alpha^3$ and~$\beta^2$ act minimally. Lemma~\ref{Lemm:DescentSigmaDp4} below shows that~$X_4$ and $Y_4$ are del Pezzo surfaces of degree~$4$ and  both $\Pic{X_4}^{\alpha^3}$ and $\Pic{Y_4}^{\beta^2}$ have rank $2$ and are generated by the fibres of two conic bundles on $X_4$ and $Y_4$. Choosing a birational map $X_4\dasharrow Y_4$ conjugating $\alpha^3$ to $\beta^2$ (which exists if and only if the elliptic curves are isomorphic), which is general enough, we should obtain an embedding of $\mathrm{SL}(2,\mathbb{Z})$ such that any element of infinite order is hyperbolic.

In order to prove that there is no more relation in the group generated by $\alpha$ and $\beta$ and that all elements of infinite order are hyperbolic, we describe the morphisms $X\to X_4$ and~$Y\to Y_4$ and the action of $\alpha$ and $\beta$ on $\Pic{X}^{\alpha^3}$ and $\Pic{Y}^{\beta^2}$ (which are generated by the fibres of the two conic bundles on $X_4$, and $Y_4$ and by the exceptional curves obtained by blowing-up points on the elliptic curves fixed), and then observe that the composition of the elements does what is expected. 

\subsection{Technical results on automorphisms of del Pezzo surfaces of degree $4$}
Recall some classical facts about del Pezzo surfaces, that the reader can find in \cite{Dem} (see also \cite{Ma}). A del Pezzo surface is a smooth projective surface $Z$ such that the anti-canonical divisor $-K_Z$ is ample. These are $\p^1\times \p^1$, $\p^2$ or $\p^2$ blown-up at $1\le r\le 8$ points in general position (no $3$ collinear, no $6$ on the same conic, no $8$ on the same cubic singular at one of the $8$ points). The degree of a del Pezzo surface $Z$ is $(K_Z)^2$, which is $8$ for $\p^1\times\p^1$, $9$ for $\p^2$ and $9-r$ for the blow-up of $\p^2$ at $r$ points.

Any del Pezzo surface $Z$ contains a finite number of $(-1)$-curves (smooth curves isomorphic to $\p^1$ and of self-intersection $-1$), each of these can be contracted to obtain another del Pezzo surface of degree $(K_Z)^2+1$. These are moreover the only irreducible curves of $Z$ of negative self-intersection. If $Z$ is not $\p^2$, there is a finite number of conic bundles $Z\to \p^1$ (up to automorphism of $\p^1$), and each of them has exactly $8-(K_Z)^2$ singular fibres. This latter fact can be find by contracting one component in each singular fibre, which is the union of two $(-1)$-curves, obtaining a line bundle on a del Pezzo surface, isomorphic to $\p^1\times\p^1$ or $\mathbb{F}_1$ and having degree $8$.

\begin{lemm}\label{Lemm:DescentSigmaDp4}
Let $Z$ be a del Pezzo surface, and let $\sigma\in \Aut(Z)$ be an involution that fixes $($pointwise$)$ an elliptic curve. Denote by $\eta\colon Z\to Z_4$ any $<\sigma>$-invariant birational morphism such that the action on $Z_4$ is minimal.

Then, $Z_4$ is a del Pezzo surface of degree~$4$, and $\Pic{Z_4}^{\sigma}=\mathbb{Z}f_1\oplus \mathbb{Z}f_2$, where $f_1,f_2$ correspond to the fibres of the two conic bundles $\pi_1$, $\pi_2\colon Z_4\to \p^1$ $($defined up to automorphism of $\p^1)$ that are invariant by $\sigma$. Moreover 
\begin{align*}
& f_1+f_2=-K_{Z_4},&& f_1\cdot f_2=2 && \text{and} &&\Pic{Z}^{\sigma}=\mathbb{Z}\eta^{*}(f_1)\oplus \mathbb{Z}\eta^{*}(f_2)\oplus \mathbb{Z}E_1\oplus\dots\oplus \mathbb{Z}E_r
\end{align*}
where $E_1,\dots,E_r$ are the $r$ irreducible curves contracted by $\eta$ $($in particular, $\eta$ only contracts invariant $(-1)$-curves$)$.
\end{lemm}

\begin{proof}
Since $Z$ is a del Pezzo surface, $Z_4$ is also a del Pezzo surface. As $\sigma$ acts minimally on $Z_4$ and fixes an elliptic curve, we have the following situation (\cite[Theorem 1.4]{BaBe}): there exists a conic bundle $\pi_1\colon Z_4\to \p^1$ such that $\pi_1\sigma =\pi_1$, $\sigma$ induces a non-trivial involution on each smooth fibre of $\pi_1$, and exchanges the two components of each singular fibre, which meet at one point. The restriction of $\pi_1$ to the elliptic curve is a double covering ramified over $4$ points, which implies that there are four singular fibres. The surface $Z_4$ is thus the blow-up of four points on $\mathbb{F}_1$ or $\p^1\times\p^1$, and has therefore degree~$4$. The fact that there are exactly two conic bundles $\pi_1$, $\pi_2\colon Z_4\to \p^2$ invariant by~$\sigma$, that $\Pic{Z_4}^{\sigma}$ is generated by the two fibres, that $f_1+f_2=-K_{Z_4}$ and that $f_1\cdot f_2=2$ can be checked in~\cite[Lemma~9.11]{Bla}.

It remains to observe that all points blown-up by $\eta$ are fixed by $\sigma$. If $\eta$ blows-up an orbit of at least two points of $Z_4$ invariant by $\sigma$, the points would be on the same fibre of $\pi_1$. The transform of this fibre on $Z$ would then contain a curve isomorphic to $\p^1$ and having self-intersection $\le -2$; this is impossible on a del Pezzo surface.
\qed\end{proof}

\begin{lemm}\label{Lem:dp4fibres}
For $i=1,2$, let $X_i$  be a projective smooth surface, with $K_{X_i}^2=4$, and let $\sigma_i \in\Aut(X_i)$ be an involution which fixes an elliptic curve $\Gamma_i\subset X_i$. Let $\pi_i\colon X_i\to \p^1$ be a conic bundle such that $\pi_i\sigma_i=\pi_i$ and let $F_i,G_i\subset X_i$ be two sections of $\pi_i$ of self-intersection~$-1$, intersecting transversally into one point.

Then, $X_1$, $X_2$ are del Pezzo surfaces of degree~$4$ and the following assertions are equi\-valent:
\begin{enumerate}
\item
There exists an isomorphism $\varphi\colon X_1\to X_2$ which conjugates $\sigma_1$ to $\sigma_2$, which sends $F_1,G_1$ onto $F_2$ and~$G_2$ respectively and such that $\pi_2\varphi=\pi_1$;
\item
The points of $\p^1$ whose fibres by $\pi_i$ are singular are the same for $i=1,2$, and $\pi_1(F_1\cap G_1)=\pi_2(F_2\cap G_2)$.
\end{enumerate}
\end{lemm}

\begin{proof}
For $i=1,2$, we denote by $\eta_i\colon X_i\to \mathbb{F}_1$ the birational morphism that contracts, in each singular fibre of $\pi_i$, the $(-1)$-curve that does not intersect $F_i$. The curve $\eta_i(F_i)$ is equal to the exceptional section $E$ of the line bundle $\pi\colon \mathbb{F}_1\to \p^1$, with $\pi=\pi_i\eta_i^{-1}$. Since $\eta_i(G_i)$ intersects $E$ into exactly one point, it is a section of self-intersection $3$. In particular, the four points blown-up by $\eta_i$ lie on $\eta_i(G_i)$. Contracting $E$ onto a point of $\p^2$, $\eta_i(G_i)$ becomes a conic of $\p^2$ passing through the five points blown-up by the birational morphism $X_i\to \p^2$; this implies that no $3$ are collinear and thus that $X_i$ is a del Pezzo surface of degree~$4$.

It is clear that the first assertion implies the second one. It remains to prove the converse. The second assertion implies that $\eta_1(G_1)\cap E=\eta_2(G_2)\cap E$, and this yields the existence of an automorphism of $\mathbb{F}_1$ that sends $\eta_1(G_1)$ onto $\eta_2(G_2)$ and that preserves any fibre of $\pi$. We can thus assume that $\eta_1(G_1)=\eta_2(G_2)$, which implies that the four points blown-up by $\eta_1$ and $\eta_2$ are the same. The isomorphism $\varphi$ can be chosen as~$\varphi=\eta_2^{-1}\circ \eta_1$. The map $\varphi$ conjugates $\sigma_1$ to $\sigma_2$ because, for each $i$, $\sigma_i$ is the unique involution that preserves any fibre of $\pi_i$ and exchanges the two components of each singular fibre (\emph{see} for example \cite[Lemma~9.11]{Bla}).
\qed\end{proof}

\subsection{Actions on the Picard groups of $\alpha$ and $\beta$}

We now describe the actions of $\alpha$ and $\beta$ on $\Pic{X}$ and $\Pic{Y}$. 

\begin{prop}\label{Matrices6}
Let $X$ be a del Pezzo surface of degree~$(K_X)^2<4$, and let $\alpha\in \Aut(X)$ be an automorphism of order~$6$ such that $\Pic{X}^{\alpha}=\mathbb{Z}K_X$ and such that $\alpha^3$ fixes pointwise an elliptic curve.
Let $\eta_X\colon X\to X_4$ be a birational morphism, so that $\alpha^3$ acts minimally on~$X_4$, and let $f_1$, $f_2\in \Pic{X}$ be the divisors corresponding to the two conic bundles on~$X_4$ which are invariant by $\alpha^3$ $($\emph{see} Lemma~\ref{Lemm:DescentSigmaDp4}$)$. Then, one of the following occurs:

(i)  $(K_X)^2=3$, $\eta_X$ contracts a curve $E_1$, and $\alpha$, $\alpha^2$ act on $\Pic{X}^{\alpha^3}$ as
\begin{align*}
&\left[
\begin{array}{rrrrr}
1& 1 & 1 \\
1 & 0 & 0\\
-2 & 0 & -1
\end{array}
\right] && \text{ and } &&
\left[
\begin{array}{rrrrr}
0& 1 & 0 \\
1 & 1 & 1\\
0 & -2 & -1
\end{array}
\right]
\end{align*}
relatively to the basis $(f_1,f_2,E_1)$ $($up to an exchange of $f_1,f_2)$.

(ii)  $(K_X)^2=1$, $\eta_X$ contracts $E_1,E_2,E_3$, and $\alpha$, $\alpha^2$ act on $\Pic{X}^{\alpha^3}$ as
\begin{align*}
&\left[
\begin{array}{rrrrr}
 1&3&1&1&1\\
 3&4&2&2&2\\
-2&-4&-2&-2&-1\\
-2&-4&-1&-2&-2\\
-2&-4&-2&-1&-2\end{array}
\right] && \text{ and } &&
\left[
\begin{array}{rrrrr}
4&3&2&2&2\\
3&1&1&1&1\\
-4&-2&-2&-1&-2\\
-4&-2&-2&-2&-1\\
-4&-2&-1&-2&-2\end{array}
\right]
\end{align*}
relatively to the basis $(f_1,f_2,E_1,E_2,E_3)$ $($up to a good choice of $E_1,E_2,E_3$ and an exchange of $f_1,f_2)$.
\end{prop}

\begin{proof}
Let $E\subset X$ be any $(-1)$-curve invariant by $\alpha^3$. The divisor $E+\alpha(E)+\alpha^2(E)$ is invariant by $\alpha$ and thus equi\-valent to $sK_X$ for some integer $s$. Computing the intersection with $K_X$ and the self-intersection, we obtain $-3=~s(K_X)^2$ and $-3+6(E\cdot \alpha(E))=s^2(K_X)^2$. This gives two possibilities:
\[\begin{array}{llll}
(i)&(K_X)^2=3,& s=-1,& E\cdot \alpha(E)=1\\
(ii)&(K_X)^2=1,& s=-3,& E\cdot \alpha(E)=2\\
\end{array}\]

In case $(i)$, $\eta_X$ is given by the choice of one $(-1)$-curve $E_1$ invariant by $\alpha^3$. Since $E_1\cdot \alpha(E_1)=1$, the divisor $E_1+\alpha(E_1)$ corresponds to a conic bundle on $X$ and $X_4$. Up to renumbering, we can say that $f_1=E_1+\alpha(E_1)$ and $f_2=E_1+\alpha^2(E_1)$. This means that $\alpha(E_1)=f_1-E_1$, $\alpha^2(E_1)=f_2-E_1$,  $\alpha(f_1)=f_1+f_2-2E_1$ and $\alpha(f_2)=f_1$.

In case $(ii)$, there are three curves $E_1$, $E_2$, $E_3$ contracted by $\eta_X$. We first choose $E_1$, and then choose $E_2'=\iota_B(\alpha(E_1))=-2K_X-\alpha(E_1)$ (where $\iota_B$ is the Bertini involution of the surface). 
Since $E_2'$ does not intersect $E_1$, we can contract $E_1$, $E_2'$, and another curve $E_3$ to obtain an $\alpha^3$-equivariant birational morphism $X\to X_4'$, where $X_4'$ is a del Pezzo surface of degree~$4$. This choice gives us two conic bundles $f_1',f_2'$ on $X_4'$, which we also see on $X_4$, invariant by $\alpha^3$. We now compute~$\alpha(E_3)$. We have $\alpha(E_3)\cdot E_3=2$, 
\begin{align*}
&\alpha(E_3)\cdot E_1=E_3\cdot \alpha^2(E_1)=E_3\cdot (-3K_X-E_1-\alpha(E_1))=E_3\cdot (-K_X-E_1+E_2')=1,\\
&\alpha(E_3)\cdot E_2'=E_3\cdot \alpha^2(E_2')=E_3\cdot (-2K_X-E_1)=2.
\end{align*}
This implies that $\alpha(E_3)=a f_1'+b f_2'-E_1-2E_2'-2E_3$, for some integers $a,b$. Computing the intersection with $-K_X$ we find $1=2a+2b-1-2-2=2(a+b)-5$, which means that $a+b=3$. Computing the self-intersection, we obtain that~$-1=2ab-1-4-4=4ab-9$, so $ab=2$. Up to an exchange of $f_1',f_2'$, we can assume that $a=1,b=2$, and obtain that $\alpha(E_3)=f_1'+2f_2'-E_1-2E_2'-2E_3=-2K_X-(f_1'-E_1)$. 

We now call $E_2$ the $(-1)$-curve $f_1'-E_2'$, which does not intersect $E_1$ or $E_3$. We take $f_1=f_1'$ and $f_2=f_1'+f_2'-2E_2'$, so that $f_1,f_2$ are conic bundles, with intersection $2$, and  $-K_X=f_1+f_2-E_1-E_2-E_3$. The contraction of $E_1,E_2,E_3$ is a $\alpha^3$-equivariant birational morphism $X\to X_4$ and $f_1$, $f_2$ correspond to the two conic bundles of $X_4$ invariant by $\alpha^3$. With this choice, we can compute
\begin{align*}
&\alpha(E_1)=\iota_B(E_2')=\iota_B(f_1-E_2)=-2K_X-(f_1-E_2),\\
&\alpha^2(E_1)=-3K_X-\alpha(E_1)-E_1=-K_X-(f_1-E_2)-E_1=-2K_X-(f_2-E_3),\\
&\alpha(E_3)=-2K_X-(f_1'-E_1)=-2K_X-(f_1-E_1),\\
&\alpha^2(E_3)=-3K_X-\alpha(E_3)-E_3=-K_X-(f_1-E_1)-E_3=-2K_X-(f_1-E_2).
\end{align*}
This yields the equalities $f_1=-2K_X+E_1-\alpha(E_3)$ and $f_2=-2K_X+E_3-\alpha^2(E_1)$, $E_2=\alpha^2(E_3)+2K_X-f_1$, and a straightforward computation gives, with the four equations above, $\alpha^i(f_j)$ and $\alpha^i(E_2)$ for $i,j=1,2$. 
\qed\end{proof}

\begin{prop}\label{Matrices4}
Let $Y$ be a del Pezzo surface of degree~$(K_Y)^2<4$, and let $\beta\in \Aut(Y)$ be an automorphism of order~$4$ such that $\Pic{Y}^{\beta}=\mathbb{Z}K_Y$ and that $\beta^2$ fixes pointwise an elliptic curve.
Let $\eta_Y\colon Y\to Y_4$ be a birational morphism, so that $\beta^2$ acts minimally on $Y_4$, and let $f_1$, $f_2\in \Pic{Y}$ be the divisors corresponding to the two conic bundles on $Y_4$ that are invariant by $\beta^2$ $($\emph{see} Lemma~\ref{Lemm:DescentSigmaDp4}$)$. Then, one of the following occurs:

(i) $(K_Y)^2=2$, $\eta_Y$ contracts two curves  $E_1$, $E_2$ and $\beta$ acts on $\Pic{Y}^{\beta^2}$ as
\[\left[
\begin{array}{rrrrr}
1& 2 & 1 &1\\
2 & 1 & 1 &1\\
-2 & -2 & -2& -1\\
-2 & -2 & -1&-2
\end{array}
\right]\]
relatively to the basis $(f_1,f_2,E_1,E_2)$.

(ii)  $(K_Y)^2=1$, $\eta_Y$ contracts $E_1$, $E_2$, $E_3$, and $\beta$ acts on $\Pic{Y}^{\beta^2}$ as
\[\left[
\begin{array}{rrrrr}
3& 4 & 2 & 2 & 2\\
4 & 3 & 2 & 2 & 2\\
-3 & -3 & -3  & -2 & -2\\
-3 & -3 & -2 & -3 & -2\\
-3 & -3 & -2 & -2 & -3
\end{array}
\right]\]
relatively to the basis $(f_1,f_2,E_1,E_2,E_3)$.
\end{prop}
\begin{rema}
The second case, numerically possible, does not exist (\emph{see} \cite{Do} or \cite{BlaC}).\end{rema}
\begin{proof}
Let $E\subset Y$ be any $(-1)$-curve invariant by $\beta^2$. The divisor $E+\beta(E)$ is invariant by $\beta$ and thus equivalent to $sK_Y$ for some integer $s$. Computing the intersection with $K_Y$ and the self-intersection, we obtain $-2=s(K_Y)^2$ and $-2+2(E\cdot \beta(E))=s^2(K_Y)^2$. This gives two possibilities:

\[\begin{array}{llll}
(i)&(K_Y)^2=2,& s=-1,& E\cdot \beta(E)=2\\
(ii)&(K_Y)^2=1,& s=-2,& E\cdot \beta(E)=3\\
\end{array}\]

In case $(i)$, there are two curves $E_1$, $E_2$ contracted by $\eta_Y$, and $\beta(E_i)=-K_Y-E_i$ for $i=1,2$. Moreover $f_i-E_1$ is also a $(-1)$-curve for $i=1,2$, so $\beta(f_i)=\beta(E_1)+\beta(f_i-E_1)=-K_Y-E_1-K_Y-(f_i-E_1)=-2K_Y-f_i$.

In case $(ii)$, there are three curves $E_1,E_2,E_3$ contracted by $\eta_Y$, and $\beta(E_i)=-2K_Y-E_i$ for $i=1,2,3$.
As before, we find $\beta(f_i)=-4K_Y-f_i$.
\qed\end{proof}

\subsection{Automorphisms of del Pezzo surfaces of order $6$, resp. $4$ -- description of $\alpha$ and $\beta$}

\subsubsection{Automorphisms of del Pezzo surfaces of order $6$}\label{Sec:Order6}\hspace{1mm}
We now give  explicit possibilities for the automorphism $\alpha\in \Aut(X)$ of order $6$.

\noindent {\bf Case I} 
$$X=\left\{(w:x:y:z)\in \p(3,1,1,2)\ \Big\vert\ w^2=z^3+\mu xz^4 +x^6+y^6\right\}$$
$$\alpha((w:x:y:z))=(w: x:-\omega y:z)$$
for some general $\mu\in \C$ so that the surface is smooth and where $\omega=e^{2\im \pi/3}$. The surface is a del Pezzo surface of degree~$1$, and $\alpha$ fixes pointwise the elliptic curve given by $y=0$. When $\mu$ varies, all possible elliptic curves are obtained. The rank of $\Pic{X}^{\alpha}$ is $1$ (\emph{see} \cite[Corollary 6.11]{Do}).

\medskip

\noindent {\bf Case II} 
$$X=\left\{(w:x:y:z)\in \p^3\ \Big\vert\ wx^2+w^3+y^3+z^3+\mu wyz=0\right\},$$
$$\alpha((w:x:y:z))=(w:-x: \omega y: \omega^2 z),$$
where $\mu\in \C$ is such that the cubic surface is smooth. The surface is a del Pezzo surface of degree~$3$, $\alpha^3$ fixes pointwise the elliptic curve given by $x=0$, and $\alpha$ acts on this via a translation of order $3$. When $\mu$ varies, all possible elliptic curves are obtained. The rank of $\Pic{X}^{\alpha}$ is $1$ (\emph{see} \cite[Page 79]{Do}).

\medskip

\noindent {\bf Case III} 
$$X=\left\{(w:x:y:z)\in \p^3\ \Big\vert\ w^3+x^3+y^3+(x+\mu y)z^2=0\right\},$$
$$\alpha((w:x:y:z))=(\omega w:x:y:-z),$$

where $\mu\in \C$ is such that the cubic surface is smooth. The surface is a del Pezzo surface of degree~$3$, $\alpha^3$ fixes pointwise the elliptic curve given by $z=0$, and $\alpha$ acts on it via an automorphism of order $3$ with $3$ fixed points. When $\mu$ varies the birational class of $\alpha$ changes (because the isomorphism class of the curve fixed by $\alpha^2$ changes) but not the isomorphism class of the elliptic curve fixed by $\alpha^3$. The rank of $\Pic{X}^{\alpha}$ is $1$ (\emph{see} \cite[Page 79]{Do}).

\subsubsection{Automorphisms of del Pezzo surfaces of order $4$}\label{Sec:Order4}\hspace{1mm}
We now give  explicit possibilities for the automorphism $\beta\in \Aut(Y)$ of order $4$.
$$Y=\left\{(w:x:y:z)\in \mathbb{P}(2,1,1,1)\ \left\vert\ w^2-x^4=\prod_{i=1}^4yz(y+z)(y+\mu z)=0\right.\right\}$$
$$\beta((w:x:y:z))=(w:\im x:y:z),$$

where $\mu\in \C\backslash \{0,1\}$. The surface is a del Pezzo surface of degree~$2$ and $\beta$ fixes pointwise the elliptic curve given by~$x=0$. When $\mu$ varies, all possible elliptic curves are obtained. The rank of $\Pic{Y}^{\beta}$ is $1$ (\emph{see} \cite[last line of page 67]{Do} or \cite{BlaC}).

There are other possibilities of automorphisms $\beta$ of order $4$ of rational surfaces $Y$ such that $\beta^2$ fixes an elliptic curve, but none for which the rank of $\Pic{Y}^\beta$ is $1$ (\emph{see}  \cite{BlaC}).

\subsection{The map $X_4\dasharrow Y_4$ that conjugates $\alpha^3$ to $\beta^2$}\label{Subsec:Themapwhichconjugates}
We now fix $\alpha\in \Aut(X)$, $\beta\in \Aut(Y)$, automorphisms of order~$6$ and $4$ respectively, which act minimally on del Pezzo surfaces $X$ and $Y$, so that $\alpha^3$ and $\beta^2$ fix (pointwise) elliptic curves $\Gamma_X\subset X$ and $\Gamma_Y\subset Y$, which are isomorphic (as abstract curves).

We denote by $\eta_X\colon X\to X_4$ and $\eta_Y\colon Y\to Y_4$ two birational morphisms to del Pezzo surfaces of degree~$4$, so that~$\alpha^3$ and $\beta^2$ act minimally on $X_4$ and $Y_4$ respectively. We denote by $f_1$, $f_2\in \Pic{X_4}\subset \Pic{X}$, respectively by $f_1'$, $f_2'\in \Pic{Y_4}\subset \Pic{Y}$,  the two divisors corresponding to the two conic bundles invariant by $\alpha^3$, respectively by $\beta^2$.

We will choose two points $q_1$, $q_2\in \eta_X(\Gamma_X)\subset X_4$, and denote by $\tau\colon Z_4\to X_4$ the blow-up of these two points. 

\begin{lemm}\label{LemmConj}
For some good choice of $q_1$, $q_2$, there exists a birational morphism $\tau'\colon Z_4 \to~Y_4$ satisfying the following properties:
\begin{enumerate}
\item
the morphism $\tau'$ is the contraction of the strict transforms of the two irreducible curves equivalent to $f_1$ passing through $q_1$ and $q_2$ onto two points $q_1',q_2'\in \eta_Y(\Gamma_{Y})$;
\item
the map $\varphi=\tau'\tau^{-1}$ conjugates $\alpha^3$ to $\beta^2$ $($i.e. $\varphi \alpha^3=\beta^2\varphi)$;
\item
neither $q_1$ nor $q_2$ is blown-up by $\eta_X$, and neither $q_1'$ nor $q_2'$ is blown-up by $\eta_Y$;
\item 
identifying $f_1,f_2$ with $\tau^{*}(f_1),\tau^{*}(f_2)\in\Pic{Z_4}$ and  $f_1',f_2'$ with $\tau'^{*}(f_1'),\tau'^{*}(f_2')\in\Pic{Z_4}$, we have the following relations in $\Pic{Z_4}$:
\begin{align*}
&f_1=f_1', && f_1'=f_1,\\
&f_2=f_2'+2f_1'-2E_{\tau'}, &&f_2'=f_2+2f_1-2E_\tau,\\
&E_{\tau}=2f_1'-E_{\tau'}, &&E_{\tau'}=2f_1-E_{\tau},
\end{align*} 
where $E_\tau$, $E_{\tau'}\in \Pic{Z_4}$ correspond to the exceptional divisors of $\tau$ and $\tau'$ respectively, which are the sum of two exceptional curves.
\end{enumerate}
\end{lemm}

\begin{proof}
Denote by $\pi\colon X_4\to \p^1$ and $\pi'\colon Y_4\to \p^1$ the morphisms whose fibres are $f_1$ and~$f_1'$ respectively. As it was already observed in the proof of Lemma~\ref{Lemm:DescentSigmaDp4}, both $\pi$, $\pi'$ are conic bundles, with four singular fibres, and the four singular fibres correspond to the four branch points of the double coverings $\pi\colon \eta_X(\Gamma_X)\to \p^1$ and $\pi'\colon \eta_Y(\Gamma_Y)\to \p^1$. Since $\Gamma_X$ and $\Gamma_Y$ are isomorphic elliptic curves, we can assume that the four points are the same for both morphisms. Denote by~$\Delta\subset \p^1$ the union of the image by $\pi$ of the points blown-up by $\eta_X$, the image by $\pi'$ of the points blown-up by $\eta_Y$, and the points corresponding to singular fibres of $\pi$ (or $\pi'$).

We define a closed subset $V\subset \Gamma_X\times \Gamma_X$ consisting of pairs $(q_1,q_2)$ that we "do not want", and denote by $U$ its complement. The closed subset $V$ is the union of the pairs $(q_1,q_2)$ such that $\pi(q_1)$ or $\pi(q_2)$ belongs to $\Delta$. Observe that $V$ is a finite union of curves of $\Gamma_X\times\Gamma_X$ (of bidegree~$(0,1)$ or $(1,0)$).

Choosing $(q_1,q_2)\in U$, such that $q_1,q_2$ are on distinct fibres of $\pi$, we can define a birational morphism $\tau'\colon Z_4\to W$ which contracts the strict transforms of the fibres of $\pi$ which pass through $q_1$ and $q_2$. The map $\varphi=\tau'\tau^{-1}$ conjugates~$\alpha^3$ to a biregular automorphism of $W$, which preserves any fibre of the conic bundle $\pi_W=\pi\varphi^{-1}$. In fact, $\varphi$ is a sequence of two elementary links of conic bundles. 
 It remains to show that for a good choice of $(q_1,q_2)\in U$, the triplet $(W,\pi_W, \varphi\alpha^3 \varphi^{-1})$ is isomorphic to $(Y,\pi',\beta^2)$, using Lemma~\ref{Lem:dp4fibres}.

Let $E_1\subset X_4$ be a $(-1)$-curve which is a section of $\pi$; we fix a birational morphism $\mu_X\colon X_4\to \p^2$ which contracts $E_1$ and all $(-1)$-curves lying on fibres of $\pi$ that do not intersect $E_1$, which we call $E_2,\dots,E_5$. The fibres of $\pi$ correspond to lines of $\p^2$ passing through the point $p_1=\mu_X(E_1)$, the curves equivalent to $f_2$ correspond to conics passing through~$p_2=\mu_X(E_2),\dots,p_5=\mu_X(E_5)$. For any pair $(q_1,q_2)$, we denote by $C\subset X_4$ (respectively $D\subset X_4$) the strict transform of the conic of $\p^2$ passing through $p_1,p_2,p_3,q_1,q_2$ (respectively $p_1,p_4,p_5,q_1,q_2$), and denote by $C',D'\subset W$ their strict transforms by $\varphi$. The curves $C$, $D$ are sections of $\pi$ and intersect into three points: $q_1$, $q_2$, $r\in X_4$. The cur\-ves~$C'$,~$D'$ are sections of $\pi_W$ of self-intersection $-1$, and intersect into one point, which is~$\varphi(r)\in W$. The isomorphism class of the triplet $(W,\pi_W,\varphi\alpha^3\varphi^{-1})$ is given by~$\pi_W(\varphi(r))\in~\p^1$ (Lemma~\ref{Lem:dp4fibres}), equal to $\pi(r)\in\p^1$. Fixing $q_1$, and choosing one of the two possibilities for~$r$, on the fibre given by the isomorphism class of $(Y,\pi_Y,\beta^2)$, the curves~$C$,~$D$ can be chosen as the conics passing respectively through $p_1$, $p_2$, $p_3$, $q_1$, $r$ and~$p_1$, $p_4$, $p_5$, $q_1$, $r$, so $q_2$ is uniquely defined. This gives us two irreducible curves $V_1,V_2$ of bidegree~$(1,1)$ in $\Gamma_X\times \Gamma_X$, which are thus not contained in $V$. Choosing a general point of~$V_1\cap U$, the triplet $(W,\pi_W,\varphi \alpha^3\varphi^{-1})$ is isomorphic to $(Y,\pi_Y,\beta^2)$. 

The fact that $\eta_X$ does not blow-up $q_1$ or $q_2$ and that $\eta_Y$ does not blow-up $q_1'$ or $q_2'$ is given by the fact that $\pi(q_i)=\pi'(q_i')\notin\Delta$ for $i=1,2$.

It remains to show the relations in $\Pic{Z_4}$. The equalities $f_1=f_1'$ and $E_\tau+E_{\tau'}=2f_1$ are given by the construction of $\tau$, $\tau'$. The adjunction formula, and the fact that $-K_{X_4}=f_1+f_2$, $-K_{Y_4}=f_1'+f_2'$ yields $-K_{Z_4}=f_1+f_2-E_\tau=f_1'+f_2'-E_{\tau'}$ and the remaining equalities.
\qed\end{proof}

\subsection{The hyperbolic embeddings}
Now we have the map $\varphi\colon X_4\dasharrow Y_4$ constructed in  \S\ref{Subsec:Themapwhichconjugates} above, which conjugates~$\alpha^3$ to $\beta^2$, the group generated by $\alpha$ and $\beta$ is a subgroup of the Cremona group, which is isomorphic to $\mathrm{SL}(2,\z)$ if and only if there is no other relation than the obvious $1=\alpha^6=\beta^4=\alpha^3\beta^2$ which arise by construction. We compute the action of $\alpha$, $\beta$ on $\Pic{X}$, $\Pic{Y}$, and on a surface $Z$ which dominates $X$, $Y$, where both $\alpha$, $\beta$ act. This surface exists if the group generated by the action of both maps on the elliptic curve fixed by $\alpha^3$ and $\beta^2$ is a finite subgroup of automorphisms of the curve (which is true for example when either $\alpha$ or $\beta$ fixes the curve), and if it does not exist, we can also compute the action on the limit of the Picard groups obtained.

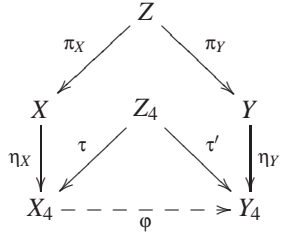
\begin{equation}\label{Eq:DiagrammFond}
\xymatrix{
& Z\ar[ld]_{\pi_X}\ar[rd]^{\pi_Y} &\\
X\ar[d]_{\eta_X} & Z_4\ar[ld]_{\tau}\ar[rd]^{\tau'} & Y\ar[d]^{\eta_Y}\\
X_4\ar@{-->}[rr]_{\varphi} & & Y_4
}\end{equation}

\begin{prop}
For $j=1,2,3$, choose $\alpha\in \Aut(X)$ as an automorphism of order $6$ of a del Pezzo surface $X$, which is respectively given in case $I$, $II$ or $III$ of  $\S\ref{Sec:Order6}$, such that $\alpha^3$  fixes pointwise an elliptic curve $\Gamma_X$,  and choose $\beta$ as an automorphism of order $4$ of a del Pezzo surface~$Y$ of degree~$2$, which fixes pointwise an elliptic curve $\Gamma_Y$ isomorphic to  $\Gamma_X$, $($which implies that $\alpha^3$ and $\beta^2$ are conjugate$)$.  This yields, with the above construction, a hyperbolic embedding $\theta_{h,j}\colon \SL(2,\z)\subset \Aut(Z)\subset \Bir(Z)\simeq\Bir(\p^2)$ which preserves an elliptic curve $\Gamma$ isomorphic to $\Gamma_X$ and $\Gamma_Y$.

The surface $Z$ is obtained by blowing-up respectively  $12$, $10$ and $10$ points on a smooth cubic curve of $\p^2$ isomorphic to $\Gamma$, and the action of $\theta_{h,i}( \SL(2,\z))$ on $\Gamma$ is respectively the identity, a translation of order $3$ and an automorphism of order $3$ with fixed point. There is no curve of $Z$ distinct from $\Gamma$ which is invariant by $\theta_{h,i}( \SL(2,\z))$. The curve $\Gamma$ can be chosen to be any elliptic curve for $j=1,2$.
\end{prop}
\begin{proof}
In case $j=1$, we take $(f_1,f_2,E_1,E_2,E_3)$ as a basis of $\Pic{X}^{\alpha^3}$, where $E_1,E_2,E_3$ are the three curves contracted by $\eta_X$, and $f_1,f_2$ correspond to the fibres of the two conic bundles invariant by $\alpha^3$ on $X_4$. Applying Proposition~\ref{Matrices6}, $\alpha$ preserves the submodule generated by~$f_1$, $f_2$, $E$, where $E=E_1+E_2+E_3$ is the divisor contracted by $\eta_X$, and its action relatively to this basis is

$$\left[\begin{array}{rrrrrrrrr}
1&3&3\\
3&4&6\\
-2&-4&-5\end{array}\right].$$

In cases $j=2,3$, we take $(f_1,f_2,E)$ as a basis of $\Pic{X}^{\alpha^3}$, where $E=E_1$ is the (irreducible) divisor contracted by $\eta_X$, and $f_1,f_2$ correspond to the fibres of the two conic bundles invariant by~$\alpha^3$ on $X_4$. Applying Proposition~\ref{Matrices6}, the action of $\alpha$ on~$\Pic{X}^{\alpha^3}$ relatively to this basis is
$$\left[
\begin{array}{rrrrr}
0& 1 & 0 \\
1 & 1 & 1\\
0 & -2 & -1
\end{array}
\right]$$
(for a good choice of $f_1,f_2,E$).

In each of the three cases, we take $(f_1',f_2',E_1',E_2')$ as a basis of $\Pic{Y}^{\beta^2}$, where $E_1',E_2'$ are the divisors contracted by $\eta_Y$, and $f_1',f_2'$ correspond to the fibres of the two conic bundles invariant by $\beta^2$ on $Y_4$. Applying Proposition~\ref{Matrices4},  $\beta$ preserves the submodule generated by $f_1',f_2',E'$, where $E'=E_1'+E_2'$ is the divisor contracted by $\eta_Y$ and its action relatively to this basis is
\[\left[
\begin{array}{rrrrr}
1& 2 & 2 \\
2 & 1 & 2 \\
-2 & -2 & -3
\end{array}
\right].\]

We denote  by $\pi_X\colon Z\to X$ the blow-up of the points corresponding to the points blown-up by $\tau$ and $\eta_Y$ (\emph{see} Dia\-gram~(\ref{Eq:DiagrammFond})), and denote again their exceptional divisors by $E_\tau$ and~$E'$. Similarly, we denote by $\pi_Y\colon Z\to Y$ the blow-up of the points corresponding to the two points blown-up by $\tau'$ and $\eta_X$, and denote again their exceptional divisors by~$E_{\tau'}$ and $E$. Since $X_4$ and $Y_4$ are del Pezzo surfaces of degree~$4$, they are obtained by blowing-up $5$ points of $\p^2$, all lying on the smooth cubic being the image of~$\Gamma_X$ or $\Gamma_Y$. This implies that~$Z$ is the blow-up of $12$ points of $\p^2$ if $i=1$ and of $10$ points of $\p^2$ if $i=2,3$, all points belonging to the smooth cubic curve.  Moreover, both $\alpha$ and $\beta$ lift to automorphisms of $Z$.

We denote by the same name the pull-backs of the divisors $f_1$, $f_2$, $E$, $E'$,  $E_\tau$  on~$Z$. Recall that $E_\tau$ is the sum of two $(-1)$-curves.  The action of $\alpha$ in case $j=1$, $\alpha$ in case $j\in\{2,3\}$ and $\beta$ in each case on the subvectorspace $W$ of $\Pic{Z}\otimes \mathbb{R}$ generated by $(f_1,f_2,E,E',E_\tau)$ are respectively
$$\left[\begin{array}{rrrrrrrrr}
1&3&3&0&0\\
3&4&6&0&0\\
-2&-4&-5&0&0\\
0&0&0&1&0\\
0&0&0&0&1\end{array}\right], 
\left[\begin{array}{rrrrrrrrr}
0&1&0&0&0\\
1&1&1&0&0\\
0&-2&-1&0&0\\
0&0&0&1&0\\
0&0&0&0&1\end{array}\right]
\mbox{ and }
\left[\begin{array}{rrrrrrrrr}
 5&10& 0 & 6 & 8\\
 2& 5&  0& 2 & 4\\
 0& 0& 1& 0& 0\\
-2&-6& 0& -3& -4\\
-4&-8& 0&-4 & -7\end{array}\right]
$$
relatively to this basis. The first two matrices are obtained because $\alpha$ fixes the curve~$\Gamma_X$, and because $E'$, $E_\tau$ correspond to points of $\Gamma_X$ which are not blown-up by $\eta_X$ (Lemma~\ref{LemmConj}). The second matrix is obtained applying again Lemma~\ref{LemmConj}, which yields the equations $f_1=f_1'$, $f_2=f_2'+2f_1'-2E_{\tau'}$, $E_\tau=2f_1'-E_{\tau'}$. One easily checks that the only elements of~$W$ which are fixed by $\alpha$ and $\beta$ are the multiples of the canonical divisor, correspon\-ding to $[1,1,-1,-1,-1]$. This implies that any curve $C\subset Z$ invariant by the group is a multiple of the elliptic curve  $\Gamma_Z\subset Z$ (strict transform of $\Gamma_X$ and $\Gamma_Y$). This curve having negative self-intersection, $C$ has to be equal to $\Gamma_Z$.

By construction, we have $\alpha^6=\beta^4=1$ and $\beta^2=\alpha^3$. We have to prove that no other relation holds, and that any element of infinite order corresponds to a hyperbolic element of $\Aut(Z)$. Writing $\rho_1=\alpha\beta$ and~$\rho_2=\alpha^2\beta$, this corresponds to show that for any sequence $(i_1,\dots,i_n)$ with $i_k\in \{1,2\}$, the element $\rho_{i_n}\cdot \dots \cdot\rho_{i_1}$ is a hyperbolic element of $\Aut(Z)$.

To show this, we look at the action of $\alpha,\beta$ on the orthogonal $W_0=K^\perp$ of the canonical divisor $K\in W\subset \Pic{Z}$ in $W$. We choose a basis of $W_0$, made of orthogonal eigenvectors of $\beta$.

If $j=1$, the basis is 
$<[1,0,0,-1,0],$ $[2,1,0,-1,-2],$ $[3,1,-2,-1,-2]$, $[4,2,-2,-2,-3]>$, which has signature $<-2,-2,-2,2>$ and the actions of $\alpha,\alpha^2,\beta$ relatively to it are respectively
$$\left[\begin{array}{rrrr}0&-1&-2&-2\\-2&-2&-3&-4\\-1&0&-2&-2\\2&2&4&5\end{array}\right], \left[\begin{array}{rrrr}0&-2&-1&-2\\-1&-2&0&-2\\-2&-3&-2&-4\\2&4&2&5\end{array}\right], \left[\begin{array}{rrrr}-1&0&0&0\\0&-1&0&0\\0&0&1&0\\0&0&0&1\end{array}\right].$$
We denote by $H$ the fourth basis vector, which is the only one with positive square, and compute by induction on $n$ the vector $H_n=\rho_{i_n}\cdot \dots \cdot\rho_{i_1}(H)$ for $n\ge 0$ (with $H_0=H$). Writing $H_n=\left[\begin{array}{r}-a_n\\ -b_n\\ -c_n\\ \ell_n\end{array}\right]$, we prove by induction on $n$ the following inequalities:
\begin{equation}\label{IneqInduc}
 \begin{array}{rcl}
 a_n,\,b_n,\,c_n,\,\ell_n&\ge& 0\\
 \ell_n&>&\frac{6}{5}c_n\\
 \ell_n&>&2a_n\\
 \ell_n&\ge &(\frac{5}{3})^n,\end{array}\end{equation}
where the last one will yield the result, implying that $\rho_{i_k}\cdot \dots\cdot \rho_{i_1}$  is a hyperbolic element of $\Aut(Z)$ of dynamical degree~$\ge (\frac{5}{3})^k$.

Note that $(\ref{IneqInduc})$ is easily checked for $n=0$, since $\ell_0=1$, $a_0=b_0=c_0=0$. We assume the result true for $n$ and prove it for $n+1$. We have $H_{n+1}=\rho_{i_{n+1}}(H_{n})=\alpha^{i_{n+1}}\beta(H_{n})$, which is equal to 

$$\left[\begin{array}{r}-b_{n}+2c_n-2\ell_n\\ -2a_n-2b_n+3c_n-4\ell_n\\ -a_n+2c_n-2\ell_n\\ 2a_n+2b_n-4c_n+5\ell_{n}\end{array}\right]\mbox{ or }
\left[\begin{array}{r}-2b_n+c_n-2\ell_n\\ -a_n-2b_n-2\ell_n\\ -2a_n-3b_n+2c_n-4\ell_n\\ 2a_n+4b_n-2c_n+5\ell_n\end{array}\right].$$

 We deduce  the inequalities  $a_{n+1}$, $b_{n+1}$, $c_{n+1}$, $\ell_{n+1}\ge 0$ directly from $a_n$, $b_n\ge 0$ and $\ell_n\ge c_n\ge 0$. Computing $\ell_{n+1}-2a_{n+1}=\ell_n+2a_{2n}$, we obtain $\ell_{n+1}>2a_{n+1}$. We compute then $5\ell_{n+1}-6c_{n+1}$ to see that it is positive,  and obtain either $13\ell_n-8c_n+4a_n+10b_n>(13-8\cdot \frac{5}{6})\ell_n+4a_n+10b_n>0$ or $\ell_n+2c_n+2b_n-2a_n>0$. To get (\ref{IneqInduc}), it remains to see that
$\ell_{n+1}\ge 5\ell_n-4c_n=\frac{5}{3}\ell_n+ 4(\frac{5}{6}\ell_n-c_n)>\frac{5}{3}\ell_n\ge (\frac{5}{3})^{n+1}.$\\

For  $j=2,3$, the situation is similar, with other data. The basis is now
$<[1,0,0,-1,0],$ $[2,1,0,-1,-2],$ $[8,2,-2,-2,-5]$, $[9,3,-2,-3,-6]>$, which has signature $<-2,-2,-6,6>$ and the actions of $\alpha,\alpha^2,\beta$ relatively to it are respectively
$$\left[\begin{array}{rrrr}-2&-9&-18&-24\\-6&-20&-36&-51\\-6&-18&-35&-48\\7&22&42&58\end{array}\right], \left[\begin{array}{rrrr}-2&-6&-18&-21\\-9&-20&-54&-66\\-6&-12&-35&-42\\8&17&48&58\end{array}\right], \left[\begin{array}{rrrr}-1&0&0&0\\0&-1&0&0\\0&0&1&0\\0&0&0&1\end{array}\right].$$
We again denote by $H$ the fourth basis vector, which is the only one with positive square, and compute by induction on $n$ the vector $H_n=\rho_{i_n}\cdot \dots \cdot\rho_{i_1}(H)$ for $n\ge 0$ (with $H_0=H$). Writing $H_n=\left[\begin{array}{r}-a_{n}\\ -b_{n}\\ -c_{n}\\ \ell_{n}\end{array}\right]$, we prove by induction on $n$ the following inequalities:
\begin{equation}\label{IneqInduc2}
 \begin{array}{rcl}
 a_{n},b_{n},c_{n},\ell_n&\ge& 0\\
 \ell_n&>&c_{n}\\
 \ell_n&\ge &10^n,\end{array}\end{equation}
where the last one will yield the result, implying that $\rho_{i_k}\cdot \dots\cdot \rho_{i_1}$  is a hyperbolic element of $\Aut(Z)$ of dynamical degree~$\ge 10^k$.

Again, $(\ref{IneqInduc2})$ is easily checked for $n=0$, since $\ell_0=1$, $a_0=b_0=c_0=0$. We assume the result true for $n$ and prove it for $n+1$. We have $H_{n+1}=\rho_{i_{n+1}}(H_{n})=\alpha^{i_{n+1}}\beta(H_{n})$, which is equal to 

$$\left[\begin{array}{r}-2a_n-9b_n+18c_n-24\ell_n\\
 -6a_n-20b_n+36c_n-51\ell_n\\ -6a_n-18b_n+35c_n-48\ell_n\\
  7a_n+22b_n-42c_n+58\ell_n\end{array}\right]\mbox{ or }
\left[\begin{array}{r}-2a_n-6b_{n}+18c_n-21\ell_n\\
 -9a_n-20b_n+54c_n-66\ell_n\\ 
 -6a_n-12b_n+35c_n-42\ell_n\\
  8a_n+17b_n-48c_n+58\ell_n\end{array}\right].$$

 We deduce  the inequalities  $a_{n+1}$, $b_{n+1}$, $c_{n+1}$, $\ell_{n+1}\ge 0$ directly from $a_n$, $b_n\ge 0$ and $\ell_n\ge c_n\ge 0$. Since $\ell_{n+1}-c_{n+1}$ is either equal to $a_n+4b_n-7b_n+10\ell_n$ or to $2a_n+5b_n-13c_n+16\ell_n$, it is positive. To get (\ref{IneqInduc2}), it remains to see that
$$\ell_{n+1}\ge 58\ell_n-48c_n=10 \ell_n+48(\ell_n-c_n)\ge 10\ell_n\ge (10)^{n+1}.$$
\qed
\end{proof}

\end{document}